\documentclass{amsart}
\usepackage{amsmath}
\usepackage{amssymb}
\usepackage{amsthm}
\usepackage{enumerate}
\usepackage[pdftex]{graphicx}
\usepackage{caption}
\theoremstyle{definition}
\newtheorem{definition}{Definition}[section]
\theoremstyle{plain}
\newtheorem{lemma}[definition]{Lemma}
\newtheorem{theorem}[definition]{Theorem}

\newtheorem{proposition}[definition]{Proposition}
\newtheorem{corollary}[definition]{Corollary}
\theoremstyle{remark}
\newtheorem{remark}[definition]{Remark}

\makeatletter
\@namedef{subjclassname@2020}{
  \textup{2020} Mathematics Subject Classification}
\makeatother

\newcommand{\mycl}{\operatorname{cl}}
\newcommand{\myint}{\operatorname{int}}
\newcommand{\mydist}{\operatorname{dist}}
\newcommand{\mysupp}{\operatorname{supp}}

\newcommand{\mymult}{\operatorname{mult}}
\newcommand{\myinv}{\operatorname{inv}}

\begin{document}
\title[Approximation and imbedding theorem]{Approximation of definable functions in a definably complete locally o-minimal structure and imbedding theorem}
\author[M. Fujita]{Masato Fujita}
\address{Department of Liberal Arts,
Japan Coast Guard Academy,
5-1 Wakaba-cho, Kure, Hiroshima 737-8512, Japan}
\email{fujita.masato.p34@kyoto-u.jp}

\author[T. Kawakami]{Tomohiro Kawakami}
\address{Department of Mathematics,
	Wakayama University,
	Wakayama, 640-8510, Japan}
\email{kawa0726@gmail.com}

\begin{abstract}
We consider a definably complete locally o-minimal expansion of an ordered field.
We treat two topics in this paper.
The first topic is a definable $\mathcal C^r$ approximation of a definable $\mathcal C^{r-1}$ map between definable $\mathcal C^r$ submanifolds.
The second topic is the imbedding theorem for definable $\mathcal C^r$ manifolds.
We demonstrate that a definably normal definable $\mathcal C^r$ manifold is a definably $\mathcal C^r$ diffeomorphic to a definable $\mathcal C^r$ submanifold.
It enables us to show that the definable quotient of a definable $\mathcal C^r$ group by a definable subgroup exists when the group is bounded and closed in the ambient space.   
\end{abstract}



\maketitle

\section{Introduction}\label{sec:intro}
Locally o-minimal structures, which are generalizations of o-minimal structures \cite{vdD, vdDM, KPS, PS}, were proposed in \cite{TV}.
They are recently studied in \cite{F, S, KTTT, Fuji, Fuji3, Fuji4, Fuji5, Fuji6, FK, FKK}. 
The early study on locally o-minimal structures by Fornasiero \cite{F} assumes that the structure is an expansion of an ordered field, but the succeeding studies do not employ this assumption except \cite{FK}.

We work with a definably complete locally o-minimal expansion of an ordered field in this paper like Fornasiero.
Cell decomposition and stratification are often used to investigate features of sets and maps definable in o-minimal structures.
The first author announced that, using the decomposition theorem into special submanifolds with tubular neighborhoods (Theorem \ref{thm:tubular_decom}) instead of them, we can prove several assertions similar to the o-minimal counterparts in \cite{Fuji5}.
We recall the results in the previous studies including the above theorem in Section \ref{sec:preliminary}.
%In Section \ref{sec:zero}, we demonstrate that a definable closed subset is the zero set of a definable $\mathcal C^r$ function.
%Theorem \ref{thm:zeroset} is our main theorem.
%We construct a definable $\mathcal C^{r-1}$ tubular neighborhood of a definable $\mathcal C^r$ submanifold of $F^n$ using the previous result in Section \ref{sec:tub}.
We construct a definable $\mathcal C^r$ approximation of a definable $\mathcal C^{r-1}$ map between definable $\mathcal C^r$ submanifolds in Section \ref{sec:appro} using Theorem \ref{thm:tubular_decom}.
%We introduce the Positivstellensatz for definable $\mathcal C^r$ functions in Section \ref{sec:positivstellensatz}.

We treat another topic in Section \ref{sec:manidfolds}.
We propose a new definition of definable $\mathcal C^r$ manifolds and demonstrate that it is definably $\mathcal C^r$ imbeddable into $F^n$ for some positive integer $n$ in Section \ref{sec:manidfolds}, where $F$ is the underlying space of the given structure.
Using this result, we demonstrate that a definable quotient of a definable $\mathcal C^r$ group by a definable subgroup exists in the same section when the group is bounded and closed in the ambient space.

In the last of this section, we summarize the notations and terms used in this paper.
The term `definable' means `definable in the given structure with parameters' in this paper.
We treat definable $\mathcal C^r$ manifolds and definable $\mathcal C^r$ maps etc. in this paper.
They are called $\mathcal D^r$ manifolds and $\mathcal D^r$ maps etc. for short. 
For a linearly ordered structure $\mathcal F=(F,<,\ldots)$, an open interval is a definable set of the form $\{x \in F\;|\; a < x < b\}$ for some $a,b \in F \cup \{\pm \infty\}$.
It is denoted by $(a,b)$ in this paper.
We define a closed interval similarly. 
It is denoted by $[a,b]$.
An open box is the Cartesian product of nonempty open intervals and a closed box is defined similarly.
For a sequence $\alpha=(\alpha_1,\ldots,\alpha_n)$ of nonnegative integers, the notation $|\alpha|$ denotes the sum $\sum_{i=1}^n \alpha_i$.
This symbol also represents the absolute value of an element in an ordered abelian group.
This abuse of notation will not confuse readers.
Let $A$ be a subset of a topological space.
The notations $\myint(A)$, $\mycl(A)$ and $\partial A$ denote the interior, the closure and the frontier of the set $A$, respectively.

\section{Preliminary}\label{sec:preliminary}
We first recall the definitions of local o-minimality and definably completeness.
\begin{definition}[\cite{M,TV}]
An expansion of a dense linear order without endpoints $\mathcal F=(F,<,\ldots)$ is \textit{locally o-minimal} if, for every definable subset $X$ of $F$ and for every point $a\in F$, there exists an open interval $I$ containing the point $a$ such that $X \cap I$ is  a finite union of points and open intervals.
We can immediately show that $X \cap I$ is the union of at most one point and at most two open intervals if we choose $I$ appropriately.
The expansion $\mathcal F$ is \textit{definably complete} if any definable subset $X$ of $F$ has the supremum and  infimum in $F \cup \{\pm \infty\}$.
\end{definition}

The definition of dimension is found in \cite[Definition 3.1]{Fuji4}.

\begin{definition}[Dimension]\label{def:dim}
Consider an expansion of a densely linearly order without endpoints $\mathcal F=(F,<,\ldots)$.
Let $X$ be a nonempty definable subset of $F^n$.
The dimension of $X$ is the maximal nonnegative integer $d$ such that $\pi(X)$ has a nonempty interior for some coordinate projection $\pi:F^n \rightarrow F^d$.
We consider that $F^0$ is a singleton with the trivial topology.
We set $\dim(X)=-\infty$ when $X$ is an empty set.
\end{definition}

The following proposition summarizes the results in the previous studies.
\begin{proposition}\label{prop:pre}
Let $\mathcal F=(F,<,+,0,\cdot,1,\ldots)$ be a definably complete locally o-minimal expansion of an ordered field.
The following assertions hold true.
\begin{enumerate}
\item[(1)] A definable set is of dimension zero if and only if it is discrete.
A definable set of dimension zero is always closed.
\item[(2)] Let $X,Y \subset F^n$ be definable sets. We have $\dim (X \cup Y)=\max \{\dim X, \dim Y\}$.
\item[(3)] Let $X$ be a definable subset of $F^n$. We have $\dim \partial X < \dim X$ and $\dim \mycl(X) = \dim X$.
\item[(4)] 
Let $f: X \rightarrow F$ be a definable function.
The notation $D^0(f)$ denotes the set of points at which $f$ is not continuous.
We have $\dim D^0(f) < \dim X$.
\item[(5)] Let $f:X \rightarrow F^n$ be a definable map. We have $\dim f(X) \leq \dim X$.
\item[(6)] Let $\varphi:X \rightarrow Y$ be a definable surjective map whose fibers are equi-dimensional; that is, the dimensions of the fibers $\varphi^{-1}(y)$ are constant.
We have $\dim X = \dim Y + \dim \varphi^{-1}(y)$ for all $y \in Y$.  
\item[(7)] Let $X$ be a definable subset of $F^n$.
There exists a point $x \in X$ such that the equality $\dim(X \cap B)=\dim(X)$ is satisfied for any open box $B$ containing the point $x$.
%\item[(8)] There exists a finite partition of $X \subseteq F^n$ into special submanifolds. The number of special submanifolds is bounded by a function of $n$.
\item[(8)] For any definable subset $X$ of $F$, there exists $r \in F$ such that either the interval $(r,\infty)$ is contained in $X$ or $X$ has an empty intersection with the interval $(r,\infty)$.
\item[(9)] Let $I$ be an interval and $f:I \rightarrow F$ be a definable function.
There exists a mutually disjoint definable partition $I=X_d \cup X_c \cup X_+ \cup X_-$ satisfying the following conditions:
\begin{enumerate}
\item[(a)] the definable set $X_d$ is discrete and closed;
\item[(b)] the definable set $X_c$ is open and $f$ is locally constant on $X_c$;
\item[(c)] the definable set $X_+$ is open and $f$ is locally strictly increasing and continuous on $X_+$;
\item[(d)] the definable set $X_-$ is open and $f$ is locally strictly decreasing and continuous on $X_-$.
\end{enumerate}
\end{enumerate}
\end{proposition}
\begin{proof}
We can find the assertions except (8) in \cite{FKK}.
We prove the assertion (8).
Let $X$ be a definable subset of $F$.
Consider the definable map $\varphi:M \rightarrow M$ given by $x \mapsto \dfrac{x}{1+x^2}$.
The intersection $(-u,u) \cap \varphi(X)$ is a finite union of points and open intervals by local o-minimality if we choose a sufficiently small $u>0$.
There exists $v>0$ such that the intersection $(0,v) \cap \varphi(X)$ is either empty or the interval $(0,v)$.
The positive element $r=\varphi^{-1}(v)$ satisfies the required condition.
\end{proof}

\begin{proposition}\label{prop:cr_pre}
	Let $\mathcal F=(F,<,+,0,\cdot,1,\ldots)$ be a definably complete locally o-minimal expansion of an ordered field.
	Let $f:X \to F$ be a definable map.
	Set $D^r(f)$ be the set of points at which $f$ is not of class $\mathcal D^r$.
	Then $\dim D^r(f)<\dim X$.
\end{proposition}
\begin{proof}
The proposition is already demonstrated in \cite[Theorem 5.11]{F} when $X$ is a definable open subset of $F^n$ for some $n$.
We consider the case in which $X$ is not open.
We assume for contradiction that $\dim D^r(f)=\dim X$.
Set $d=\dim X$.
By \cite[Lemma 3.6]{Fuji4} together with \cite[Theorem 2.5]{FKK}, there exist
\begin{itemize}
	\item a coordinate projection $\pi:F^n \to F^d$,
	\item an open box $V \subseteq \pi(X)$,
	\item a definable open subset $W$ of $M^n$ and 
	\item a definable continuous map $g: V \to D^r(f)$
\end{itemize}
such that 
\begin{itemize}
	\item $\pi(W)=V$,
	\item $X \cap W=g(V)$ and 
	\item the composition $\pi \circ g$ is the identity map on $V$.
\end{itemize} 
Apply this proposition to $g$, we may assume that $g$ is of class $\mathcal D^r$ by taking a smaller $V$ if necessary because $g$ is defined on an open box.
We may assume that $f \circ g$ is also of class $\mathcal D^r$ for the same reason.
The projection $\pi$ is obviously of class $\mathcal D^r$ and its restriction to $g(V)$ is the inverse of $g$.
Therefore, $f$ is also of class $\mathcal D^r$, which contradicts the definition of the set $D^r(f)$.
\end{proof}

We recall several assertions used in this paper.
\begin{proposition}\label{prop:limit}
	Let $\mathcal M=(M,<,0,+,\ldots)$ be a definably complete locally o-minimal expansion of an ordered group.
	Let $s>0$ and $f:(0,s) \rightarrow M^n$ be a bounded definable map.
	There exists a unique point $x \in M^n$ satisfying the following condition:
	$$
	\forall \varepsilon >0, \exists \delta>0, \forall t, \ 0<t<\delta \Rightarrow |x-f(t)| < \varepsilon \text{.}
	$$ 
	The notation $\lim_{t \to +0}f(t)$ denotes the point $x$.
\end{proposition}
\begin{proof}
	See \cite[Proposition 2.5]{Fuji5}.
\end{proof}

\begin{proposition}[Definable choice lemma]\label{prop:definable_choice}
	Consider a definably complete locally o-minimal expansion of an ordered group $\mathcal M=(M,<,0,+\ldots)$.
	Let $\pi:M^{m+n} \rightarrow M^m$ be a coordinate projection.
	Let $X$ and $Y$ be definable subsets of $M^m$ and $M^{m+n}$, respectively,  satisfying the equality $\pi(Y)=X$.
	There exists a definable map $\varphi:X \rightarrow Y$ such that $\pi(\varphi(x))=x$ for all $x \in X$.
	Furthermore, if $Y$ is a definable equivalence relation defined on a definable set $X$, there exists a definable subset $S$ of $X$ such that $S$ intersects at exactly one point with each equivalence class of $Y$. 
\end{proposition}
\begin{proof}
	See \cite[Lemma 2.8]{Fuji5} and its proof. 
\end{proof}

\begin{definition}
	Consider an expansion of a dense linear order without endpoints $\mathcal M=(M,<,\ldots)$.
	A \textit{definable curve} in $X$ is a definable continuous map on an open interval $(c,d)$ into $X$.
	We also call its image a \textit{definable curve}.
	This abuse of terminology will not confuse readers.
	In some cases, continuity is not required in the definition of definable curves, but we require it in this paper. 
	When we consider a definably complete locally o-minimal expansion of an ordered group, the curve $\gamma:(0,\varepsilon) \to X$ has at most one point $x$ in $M^m$ such that, for any $0<\delta<\varepsilon$,  any open neighborhood of $x$ in $X$ intersect with the definable curve $\gamma((0,\delta))$ by Proposition \ref{prop:limit}. 
	The notation $\lim_{t \to 0}\gamma(t)$ denotes this point if it exists.   
	The definable curve is \textit{completable in $X$} if $\lim_{t \to 0}\gamma(t)$ exists and belongs to $X$.
\end{definition}

\begin{proposition}[Curve selection lemma]\label{prop:curve_selection}
	Consider a definably complete locally o-minimal expansion of an ordered group $\mathcal M=(M,<,+,0,\ldots)$.
	Let $X$ be a definable subset of $M^n$ which is not closed.
	Take a point $a \in \partial X$.
	There exist a small positive $\varepsilon$ and a definable continuous map $\gamma:(0,\varepsilon) \rightarrow X$ such that $\lim_{t \to +0}\gamma(t)=a$.
\end{proposition}
\begin{proof}
	See \cite[Corollary 2.9]{Fuji5}. 
\end{proof}

\begin{theorem}\label{thm:zeroset}
	Consider a definably complete locally o-minimal expansion of an ordered field $\mathcal F=(F,<,+,0,\cdot,1,\ldots)$.
	A definable closed set is the zero set of a $\mathcal D^r$ function for $r<\infty$.
\end{theorem}
\begin{proof}
	A definably complete locally o-minimal structure is a d-minimal structure because a definable subset of $F$ either has a nonempty interior or is discrete by Definition \ref{def:dim} and Proposition \ref{prop:pre}(1) and any structure elementarily equivalent to $\mathcal F$ is definably complete and locally o-minimal.
	The theorem was already proven when the structure is definably complete and d-minimal.
	See \cite[2.9]{MT}.
\end{proof}

We next recall the definition of special submanifolds.
The definition below given in \cite[Definition 5.4, Definition 5.5]{F} is not the literally same as the definition of special submanifolds given in \cite[Definition 3.1]{Fuji5}, but they coincide by \cite[Proposition 3.13]{Fuji5}.
We also slightly extend the definitions to the $\mathcal C^r$ case.

\begin{definition}
Let $\mathcal F=(F,<,+,0,\cdot,1,\ldots)$ be an expansion of an ordered field.
Let $X$ be a definable subset of $F^n$ of dimension $d$ and $\pi:F^n \rightarrow F^d$ be a coordinate projection.
Let $\sigma$ be a permutation of $\{1,\ldots,n\}$ such that the composition $\pi \circ \overline{\sigma}$ is the coordinate projection onto first $d$ coordinates, where $\overline{\sigma}:F^n \rightarrow F^n$ is the map given by $\overline{\sigma}(x_1,\ldots, x_n)=(x_{\sigma(1)},\ldots,x_{\sigma(n)})$.

We first consider the case in which $\overline{\sigma}(X) \subseteq F^d \times (0,1)^{n-d}$. 
A point $(a,b) \in F^n$ is \textit{$(X,\pi)$-normal} if there exist a definable neighborhood $A$ of $a$ in $F^d$ and a definable neighborhood $B$ of $b$ in $F^{n-d}$ such that either $A \times B$ is disjoint from $\overline{\sigma}(X)$ or $(A \times B) \cap \overline{\sigma}(X)$ is the graph of a definable continuous map $f:A \rightarrow B$.
We call the point $(a,b) \in F^n$ \textit{$(X,\pi)$-$\mathcal C^r$-normal} if the function $f$ given above is a $\mathcal D^r$ map.
A point $a \in F^d$ is \textit{$(X,\pi)$-bad} if it is the projection of a non-$(X,\pi)$-normal point; otherwise, the point $a$ is called \textit{$(X,\pi)$-good}.
We define \textit{$(X,\pi)$-$\mathcal C^r$-bad} points and \textit{$(X,\pi)$-$\mathcal C^r$-good} points similarly.

If $X$ is unbounded, let $\phi:F \rightarrow (0,1)$ be a definable $C^r$ diffeomorphism.
For simplicity, we assume that $\pi$ is the projection onto the first $d$ coordinates.
We consider the other cases similarly.
Consider the map $\psi:=\operatorname{id}^d \times \phi^{n-d}:F^d \times F^{n-d} \rightarrow F^d \times (0,1)^{n-d}$.
We say that $a$ is \textit{$(X,\pi)$-good} if it is $(\psi(X),\pi)$-good.
We define $(X,\pi)$-bad points etc. similarly.

A definable subset is a \textit{$\pi$-quasi-special submanifold} or simply a \textit{quasi-special submanifold} if $\pi(X)$ is a definable open set and, for every point $x \in \pi(X)$, there exists an open box $U$ in $M^d$ containing the point $x$ satisfying the following condition:
For any $y \in X \cap \pi^{-1}(x)$, there exist an open box $V$ in $M^n$ with $y \in V$ and a definable continuous map $\tau:U \rightarrow M^n$ such that $\pi(V)=U$, $\tau(U)=X \cap V$ and the composition $\pi \circ \tau$ is the identity map on $U$.
It is known that a definable set $X$ is a $\pi$-quasi-special submanifold if and only if every point of $X$ is $(X,\pi)$-normal by \cite[Lemma 4.2]{Fuji4} and \cite[Theorem 2.5]{FKK} when the structure $\mathcal F$ is definably complete and locally o-minimal.
We define $\pi$-quasi-special $\mathcal C^r$ submanifolds in the same manner.

The definable set $X$ is a \textit{$\pi$-special submanifold} if every point of $\pi(X)$ is $(X,\pi)$-good.
We simply call it a special submanifold when the projection $\pi$ is clear from the context.
A \textit{$\pi$-special $\mathcal C^r$ submanifold} is defined similarly.
See \cite[Example 3.5]{Fuji5} for a typical example of quasi-special submanifold which is not a special submanifold.

Let $\{X_i\}_{i=1}^m$ be a finite family of definable subsets of $F$. 
A \textit{decomposition of $F^n$ into special $\mathcal C^r$ submanifolds partitioning $\{X_i\}_{i=1}^m$} is a finite family of special $\mathcal C^r$ submanifolds $\{C_i\}_{i=1}^N$ such that $\bigcup_{i=1}^N C_i =F^n$, $C_i \cap C_j = \emptyset$ when $i \neq j$ and either $C_i$ has an empty intersection with $X_j$ or is contained in $X_j$ for any $1 \leq i \leq m$ and $1 \leq j \leq N$.
A decomposition $\{C_i\}_{i=1}^N$ into special $\mathcal C^r$ submanifolds satisfies the \textit{frontier condition} if the closure of any special submanifold $C_i$ is the union of a subfamily of the decomposition.
\end{definition}

We first demonstrate that $F^n$ is partitioned into special $\mathcal C^r$ submanifolds.

\begin{proposition}\label{prop:multi1}
Let $\mathcal F=(F,<,+,0,\cdot,1,\ldots)$ be a definably complete locally o-minimal expansion of an ordered field.
Let $r \geq 0$.
Let $\{X_i\}_{i=1}^m$ be a finite family of definable subsets of $F^n$.
There exists a decomposition of $F^n$ into special $\mathcal C^r$ submanifolds partitioning $\{X_i\}_{i=1}^m$.
In addition, the number of special $\mathcal C^r$ submanifolds is bounded by a function of $m$ and $n$.
\end{proposition}
\begin{proof}
We put $X_i^0=X_i$ and $X_i^1=F^n \setminus X_i$ for all $1 \leq i \leq m$.
For any sequence $\mathbf{j}=(j_1,\ldots, j_m) \in \{0,1\}^m$ of length $m$, we set $X^{\mathbf{j}}= \bigcap_{i=1}^m X_i^{j_i}$. 
Let $\{C_i^{\mathbf{j}}\}_{i=1}^{N_{\mathbf j}}$ be a partition of $X^{\mathbf{j}}$ into special $\mathcal C^r$ submanifold.
The union $\bigcup_{ \mathbf{j} \in \{0,1\}^m}\{C_i^{\mathbf{j}}\}_{i=1}^{N_{\mathbf j}}$ is a decomposition of of $F^n$ into special $\mathcal C^r$ submanifolds partitioning $\{X_i\}_{i=1}^m$.
Therefore, we have only to partition a given definable set $X$ into special $\mathcal C^r$ submanifolds.

The proposition was already demonstrated when $r=0$ in \cite[Theorem 5.6]{F} or more generally in \cite[Theorem 3.19]{Fuji5}.
Therefore, we may assume that $X$ is a $\pi$-special submanifold, where $\pi:F^n \rightarrow F^d$ is a coordinate projection.
We assume that $\pi$ is the projection onto the first $d$ coordinate for simplicity of notations.
We demonstrate the proposition by the induction on $d=\dim X$.
It is obvious when $d=0$.
By the definition of a special submanifold, for any $x \in X$, there exists an open box $U$ containing $x$ and a definable continuous map $\xi_x:\pi(U) \rightarrow F^{n-d}$ such that $X \cap U$ is the graph of $\xi_x$.
We set 
$$D=\{x \in X\;|\; \xi_x \text{ is not of class }C^r\}.$$
For an arbitrary $x \in X$, the definable set $\pi(U \cap D)$ is of dimension $<d$ by Proposition \ref{prop:cr_pre}
We have $\dim(U \cap D)<d$ by Proposition \ref{prop:pre}(1),(6).
We finally get $\dim D<d$ by  Proposition \ref{prop:pre}(7).

Set $V=\myint(\pi(X)) \setminus \mycl(\pi(D))$.
We have $\dim \pi(X) \setminus V<d$ by Proposition \ref{prop:pre}(2),(3),(5).
It is obvious that $X'=X \cap \pi^{-1}(V)$ is a special $\mathcal C^r$ submanifold.
Set $X''=X \cap \pi^{-1}(\pi(X) \setminus V)$.
We have $\dim X''<d$ by Proposition \ref{prop:pre}(6).
We can get a partition of $X''$ by the induction hypothesis.
The union of the partition of $X''$ with $\{X'\}$ is a desired partition of $X$.

The `in addition' part is clear from the proof.
\end{proof}

\begin{proposition}\label{prop:dist}
	Let $\mathcal F=(F,<,+,\cdot,0,1,\ldots)$ be a definably complete locally o-minimal expansion of an ordered field.
	Let $\pi:F^n \to F^d$ be a coordinate projection and $X$ be a $\pi$-special manifold.
	Let $f$ be a positive definable continuous function on $X$.
	Then, the definable function $g:\pi(X) \to F$ defined by $g(t)=\inf_{x \in X \cap \pi^{-1}(t)}f(x)$ is a positive definable continuous function.
\end{proposition}
\begin{proof}
	Since the definable set $\{f(x)\;|\;x \in X \cap \pi^{-1}(t)\}$ is discrete and closed by the assumption and Proposition \ref{prop:pre}(1),(5), we can construct a definable map $h:U \to X$ such that $g(t)=f(h(t))$ for each $t \in U$ by Proposition \ref{prop:definable_choice}.
	
	Assume for contradiction that $g$ is not continuous at $t_0 \in \pi(X)$.
	There exists $\varepsilon >0$ such that, for any $\delta>0$, there exists $t \in U$ satisfying the inequalities $|t-t_0|<\delta$ and $|g(t) - g(t_0)| \geq \varepsilon$.
	By Proposition \ref{prop:definable_choice}, we can find a definable map $\gamma:(0,\infty) \to U$ such that $|\gamma(s)-t_0|<s$ and $|g(\gamma(s))-g(t_0)| \geq \varepsilon$ for every $s>0$.
	We can find $\delta>0$ so that the restrictions of $\gamma$ and $h \circ \gamma$ to $(0,\delta)$ are continuous by Proposition \ref{prop:pre}(9).
	
	We can take an open box $B$ containing the point $t_0$ and contained in $\pi(X)$ such that, for any $x \in X \cap \pi^{-1}(t_0)$, there exists an open box $V_x$ in $F^n$ containing the point $x$ such that $\pi(V_x)=U$, 
	\begin{equation}
	X \cap \pi^{-1}(U) \subseteq \bigcup_{x \in X \cap \pi^{-1}(t_0)}V_x \label{eq:bbb}
	\end{equation}
	and the restriction of $\pi$ to $X \cap V_x$ is a definable homeomorphism onto $U$ by \cite[Proposition 3.13]{Fuji5}. 
	This implies that $h \circ \gamma((0,\delta))$ is contained in $X \cap V_x$ for some $x \in X \cap \pi^{-1}(t_0)$ because  $h \circ \gamma((0,\delta))$ is definably connected.
	We fix such an $x \in X \cap \pi^{-1}(t_0)$.
	Let $\varphi:U \to X \cap V_x$ be the inverse of the restriction of $\pi$ to $X \cap V_x$.
	Observe that $h \circ \gamma(s)=\varphi \circ \gamma(s)$ for $0<s<\delta$.
	We may assume that $\varphi$ is bounded after shrinking $U$ and taking smaller $\delta$ if necessary.
	We can find the limit $\lim_{s \to 0} \varphi \circ \gamma(s)$ by Proposition \ref{prop:limit}, and obviously it coincides with $x$ because $\varphi$ is continuous.
	We have demonstrated that $x=\lim_{s \to 0} h \circ \gamma(s)$.
	We have $|f(x)-g(t_0)| \geq \varepsilon$ because $f$ is continuous and $|f \circ h \circ \gamma(s)-g(t_0)| \geq \varepsilon$ for every $0<s<\delta$.
	We also get $\pi(x)=\lim_{s \to 0} \pi \circ h \circ \gamma(s) = \lim_{s \to 0} \gamma(s)=t_0$ because $\pi$ is continuous and the inequality $|t_0-\gamma(s)|<s$ holds. 
	
	We have $f(x) \geq g(t_0)$ by the definition of $g$ and the fact $\pi(x)=t_0$.
	Combining it with the inequality $|f(x)-g(t_0)| \geq \varepsilon$, we get the inequality $f(x) \geq g(t_0)+\varepsilon$.
	If we choose $s>0$ sufficiently small, we have $|f(x)-f(h(\gamma(s)))|< \varepsilon/2$ because $x=\lim_{s \to 0} h \circ \gamma(s)$ and $f$ is continuous.
	We may also assume that $|g(t_0)-f (x_s)|=|f \circ h(t_0)-f (x_s)|<\varepsilon/2$ for some $x_s \in X$ with $\pi(x_s)=\gamma(s)$ because the restriction of $\pi$ to $X \cap V_{h(t_0)}$ is a definable homeomorphism onto $U$.
	These inequalities imply $f(h(\gamma(s)))>f(x_s)$, which contradicts the definitions of $g$ and $h$.	
\end{proof}

We recall a special submanifold with a tubular neighborhood.
\begin{definition}[\cite{Fuji5}]\label{def:tub1}
Let $\mathcal F=(F,<,+,0,\cdot,1,\ldots)$ be an expansion of an ordered field.
Let $x=(x_1,\ldots, x_m), y=(y_1,\ldots, y_m)$ be points in $F^m$.
The notation $ \mydist_m:F^m \times F^m \rightarrow F$ denotes the $\mathcal D^{\infty}$  function given by $\mydist_m(x,y)=\sum_{i=1}^n(x_i-y_i)^2$.
Set $\mathcal B_m(x,r)=\{y \in F^m\;|\;  \mydist_m(x,y) < r\}$ for all $x \in F^m$ and $r>0$.
Note that the definition of $\mathcal B_m(x,r)$ is slightly different from that in \cite{Fuji5}.
For a given coordinate projection $\pi:F^n \rightarrow F^d$, take a permutation $\sigma$ of $\{1,\ldots,n\}$ such that the composition $\pi \circ \overline{\sigma}$ is the coordinate projection onto first $d$ coordinates.
Set $X^{\pi}_u=\{x \in F^{n-d}\;|\; \overline{\sigma}^{-1}(u,x) \in X\}$ for all $u \in F^d$ and a definable subset $X$ of $F^n$.
The set $X^{\pi}_u$ depends on the choice of $\sigma$, but we only discuss the features of $X^{\pi}_u$ independent of $\sigma$ in this paper.

When $\dim X<n$, the tuple $(X,\pi,T,\eta,\rho)$ is a \textit{special $\mathcal C^r$ submanifold
 with a tubular neighborhood} if 
\begin{enumerate}
\item[(a)] $\pi:F^n \rightarrow F^d$ is a coordinate projection, where $d=\dim X$;
\item[(b)] $X$ is a $\pi$-special $\mathcal C^r$ submanifold such that $U=\pi(X)$ is a definable open set;
\item[(c)] $T$ is a definable open subset of $\pi^{-1}(U)$;
\item[(d)] $\eta:U \rightarrow F$ is a positive bounded $\mathcal D^r$ function such that, for all $u \in U$, we have 
$$T^{\pi}_u= \bigcup_{x \in X^{\pi}_u}  \mathcal B_{n-d}(x,\eta(u))$$
 and $$\mathcal B_{n-d}(x_1,\eta(u)) \cap \mathcal B_{n-d}(x_2,\eta(u)) = \emptyset$$ for all $x_1,x_2 \in X^{\pi}_u$ with $x_1 \neq x_2$;
\item[(e)] $\rho:T \rightarrow X$ is a $\mathcal D^r$ retraction such that, for any $u \in U$, we have $\rho(\pi^{-1}(u) \cap T) \subseteq \pi^{-1}(u) \cap X$ and $\rho(\overline{\sigma}^{-1}(u,y)))=\overline{\sigma}^{-1}(u,x)$ for all $x \in X^{\pi}_u$ and $y \in \mathcal B_{n-d}(x,\eta(u))$.
\end{enumerate}
When $\dim X=n$, the tuple $(X,\pi,T,\eta,\rho)$ is called a \textit{special $\mathcal C^r$ submanifold
 with a tubular neighborhood} if $X$ is open, $T=X$, $\eta \equiv 0$, and $\pi$ and $\rho$ are the identity maps on $F^n$ and $X$, respectively.
A \textit{decomposition of $F^n$ into special $\mathcal C^r$ submanifolds with tubular neighborhoods} is a finite family of special $\mathcal C^r$ submanifolds with tubular neighborhoods $\{(X_i,\pi_i,T_i,\eta_i,\rho_i)\}_{i=1}^N$ such that $\{(X_i,\pi_i)\}_{i=1}^N$ is a decomposition of $F^n$ into special $\mathcal C^r$ submanifolds.
We say that a decomposition $\{(X_i,\pi_i,T_i,\eta_i,\rho_i)\}_{i=1}^N$ of $F^n$ into  special $\mathcal C^r$ submanifolds with tubular neighborhoods partitions a given finite family of definable sets and satisfies the frontier condition if so does the decomposition into special submanifolds $\{(X_i,\pi_i)\}_{i=1}^N$.
\end{definition}

The following theorem guarantees the existence of the decomposition.

\begin{theorem}\label{thm:tubular_decom}
Let $\mathcal F=(F,<,+,0,\cdot,1,\ldots)$ be a definably complete locally o-minimal expansion of an ordered field.
Let $r$ be a nonnegative integer.
Let $\{X_i\}_{i=1}^m$ be a finite family of definable subsets of $F^n$.
There exists a decomposition of $F^n$ into special $\mathcal C^r$ submanifolds with tubular neighborhoods partitioning $\{X_i\}_{i=1}^m$ and satisfying the frontier condition.
In addition, the number of special $\mathcal C^r$ submanifolds with tubular neighborhoods is bounded by a function of $m$ and $n$.
\end{theorem}
\begin{proof}
We have demonstrated the theorem in \cite[Theorem 3.22]{Fuji5} when $r=0$ under the slightly different definition of $\mathcal B_m(x,r)$.
As the first author pointed out in \cite[Remark 3.23]{Fuji5}, the proof for the case in which $r>0$ under our definition of $\mathcal B_m(x,r)$ is almost the same.
The major difference is to use Proposition \ref{prop:multi1} instead of \cite[Theorem 3.19]{Fuji5}.
We omit the details.
\end{proof}

\section{Approximation of definable function}\label{sec:appro}
\subsection{Tubular neighborhood of definable submanifold}
We construct a $\mathcal D^{r-1}$ tubular neighborhood of a $\mathcal D^r$ submanifold.
We first recall the definition of a $\mathcal D^r$ submanifold.

\begin{definition}[Definable submanifolds] 
Let $\mathcal F=(F,<,+,0,\cdot,1,\ldots)$ be an expansion of an ordered field.
Let $r$ be a nonnegative integer.
A \textit{$\mathcal D^r$ submanifold} $M$ of dimension $d$ in $F^n$ is a definable subset of $F^n$ such that, for any point $a \in M$, there exist a definable open neighborhood $U$ of the point $a$, a definable open neighborhood $V$ of the origin in $F^n$ and a $\mathcal D^r$ diffeomorphism $\varphi:U \rightarrow V$ such that $\varphi(a)$ is the origin and $$\varphi(M \cap U)=\{(x_1,\ldots, x_n) \in V\;|\; x_{d+1}=\cdots=x_n=0\}.$$
The \textit{tangent bundle} $TM$ is the set of $(x,v) \in M \times F^n$ such that $v$ is a tangent vector to $M$ at $x$.
The \textit{normal bundle} $NM$ the set of $(x,v) \in M \times F^n$ such that  $v$ is orthogonal to the tangent space $T_xM$.
They are definable sets. 
See \cite{Coste} for instance.
\end{definition}

The following lemma is useful when we reduce to the case of closed $\mathcal D^r$ submanifold.
It is a direct corollary of Theorem \ref{thm:zeroset}.
\begin{lemma}\label{lem:closed_mani}
Consider a definably complete locally o-minimal expansion of an ordered field $\mathcal F=(F,<,+,0,\cdot,1,\ldots)$.
Let $r$ be a nonnegative integer.
A $\mathcal D^r$ submanifold of $F^m$ is $\mathcal D^r$ diffeomorphic to a closed $\mathcal D^r$ submanifold of $F^{m+1}$.
\end{lemma}
\begin{proof}
Let $M$ be a $\mathcal D^r$ submanifold of $F^m$.
By the definition of a $\mathcal D^r$ submanifold, the frontier $\partial M$ is closed.
We can take a $\mathcal D^r$ function $f:F^n \rightarrow F$ with $f^{-1}(0)=\partial M$ by Theorem \ref{thm:zeroset}.
Set $S=\{(x,t) \in F^n \times F\;|\; f(x)t=1\}$.
Consider the $\mathcal D^r$ diffeomorphism $\varphi:F^n \setminus \partial M \rightarrow S$ given by $\varphi(x)=(x,1/f(x))$.
The image $\varphi(M)$ is a closed $\mathcal D^r$ submanifold of $F^{m+1}$ which is $\mathcal D^r$ diffeomorphic to $M$.
\end{proof}

We get the following theorem:

\begin{theorem}[$\mathcal D^{r-1}$ tubular neighborhood]\label{thm:tub}
Consider a definably complete locally o-minimal expansion of an ordered field $\mathcal F=(F,<,+,0,\cdot,1,\ldots)$.
Let $r$ be a positive integer.
A closed $\mathcal D^r$ submanifold $M$ of $F^n$ has a $\mathcal D^{r-1}$-tubular neighborhood; i.e., there exists a definable open neighborhood $U$ of the zero section of $M \times \{0\}$ in the normal bundle $NM$ such that the restriction of the map given by $NM \ni (x,v) \mapsto  x+v \in F^n$ to $U$ is a $\mathcal D^{r-1}$-diffeomorphism onto a definable open neighborhood $\Omega$ of $M$ in $F^n$.
Moreover, we can take $U$ of the form $$U=\{(x,v) \in NM\;|\; \| v  \| < \varepsilon(x)\},$$ where $\varepsilon $ is a positive $\mathcal D^r$ function on $M$ and the notation $\| v\|$ denotes the Euclidean norm of $v$.
\end{theorem}
\begin{proof}
We can prove it in the same manner as \cite[Lemma 6.15]{Coste} using Lemma \ref{lem:aaa} described below in place of \cite[Lemma 6.12]{Coste}. 
\end{proof}

We need the following definition and proposition in order to demonstrate Lemma \ref{lem:aaa}.
\begin{definition}
For a set $X$, a family $\mathcal F$ of subsets of $X$ is called a \textit{filtered collection} if, for any $B_1, B_2 \in \mathcal F$, there exists $B_3 \in \mathcal F$ with $B_3 \subseteq B_1 \cap B_2$. 

Consider an expansion of a dense linear order without endpoints $\mathcal F=(F;<,\ldots)$.
Let $X$ and $T$ be definable subsets.
The parameterized family $\{S_t\}_{t \in T}$ of definable subsets of $X$ is called \textit{definable} if the union $\bigcup_{t \in T} \{t\} \times S_t$ is definable in $T \times X$.

A parameterized family $\{S_t\}_{t \in T}$ of definable subsets of $X$ is a \textit{definable filtered collection} if it is simultaneously definable and a filtered collection.
\end{definition}

\begin{proposition}\label{prop:def_compact}
Consider a definably complete expansion of a dense linear order without endpoints $\mathcal F=(F;<,\ldots)$.
Let $X$ be a closed, bounded and definable set.
Every definable filtered collection of nonempty definable closed subsets of $X$ has a nonempty intersection.
\end{proposition}
\begin{proof}
The literally same proof as that for o-minimal structures in \cite[Section 8.4]{J} works. 
\end{proof}

\begin{lemma}\label{lem:aaa}
Consider a definably complete locally o-minimal expansion of an ordered field $\mathcal F=(F,<,+,0,\cdot,1,\ldots)$.
Let $r$ be a nonnegative integer and $M$ be a closed $\mathcal D^r$ submanifold of $F^n$.
Consider a positive definable function $\psi:M \rightarrow F$ which is locally bounded from below by positive constants.
Then there exists a positive $\mathcal D^r$ function $\varepsilon:M \rightarrow F$ such that $\varepsilon < \psi$ on $M$.
\end{lemma}
\begin{proof}
For any positive $u \in F$, we set $M_u=\{x \in M\;|\; \| x  \|^2 \leq u\}$.
We can take $u_0>0$ so that $M_{u_0} \neq \emptyset$.
Consider the map $\mu:[u_0,\infty) \rightarrow F$ given by $\mu(u)=\inf\{\psi(x)\;|\; x \in M_u\}$.
The set $M_u$ is closed and bounded.
The map $\mu$ is well-defined because $\mathcal F$ is definably complete.
The map $\mu$ is definable and nonincreasing.
 We show that $\mu$ is always positive.
 Assume that $\mu(s)=0$ for some $s>u_0$ for contradiction.
 For any $t>0$, the set $Z_t=\{x \in M_s\;|\; \psi(x) \leq t\}$ is a definable closed set because $\psi$ is locally bounded from below by positive constants.
 It is nonempty by the assumption that $\mu(s)=0$.
 The family $\{Z_t\}_{t>0}$ is the definable filtered collection of nonempty definable closed subsets of $M_s$.
We can take a point $x \in \bigcap_{t>0} Z_t$ by Proposition \ref{prop:def_compact}.
We have $\psi(x) \leq 0$.
It is a contradiction to the assumption that $\psi>0$.

By Proposition \ref{prop:pre}(1),(8) and Proposition \ref{prop:cr_pre}, there exists $u_1>u_0$ such that $\mu$ is $\mathcal D^r$ on $(u_1,\infty)$.
Take $u_2 \in F$ with $u_2>u_1$.
Choose a $\mathcal D^r$ function $\theta:F \rightarrow F$ such that $\theta=0$ on $(-\infty,u_2]$, $\theta$ increases from $0$ to $1$ on $[u_2,u_2+1]$ and $\theta=1$ on $[u_2+1,\infty)$.
The function $\mu_1:F \rightarrow F$ given by $\mu_1(u)=\theta(u)\mu(u)+(1-\theta(u))\mu(u_2+1)$ is a $\mathcal D^r$ map and the map $\varepsilon: M \rightarrow F$ given by $\varepsilon(x)=\frac{1}{2}\mu_1(\|x\|^2)$ satisfies the requirement.
We prove it.

Fix $x \in M$ and put $u=\|x\|^2$.
Note that $\mu$ is a decreasing function defined on $[u_0, \infty)$ and $\mu(u) \leq \psi(x)$.
When $u \leq u_2$, we have $\varepsilon(x)=\frac{1}{2}\mu(u_2+1) \leq \frac{1}{2}\mu(u) \leq \frac{1}{2}\psi(x) <\psi(x)$ because $\mu$ is decreasing.
When $u_2 < u \leq u_2+1$, we have $\varepsilon(x)=\frac{1}{2}\{\theta(u)\mu(u)+(1-\theta(u))\mu(u_2+1)\} \leq \frac{1}{2}\{\theta(u)\mu(u)+(1-\theta(u))\mu(u)\}  =\frac{1}{2}\mu(u) \leq \frac{1}{2}\psi(x) <\psi(x)$ for the same reason.
Finally, if $u>u_2+1$, we have $\varepsilon(x)=\frac{1}{2}\mu(u)\leq \frac{1}{2}\psi(x) <\psi(x)$.
\end{proof}

\subsection{Approximation of definable function}

Theorem \ref{thm:tubular_decom} is also used to construct a $\mathcal D^{r+1}$ approximation of a $\mathcal D^r$ map. 
The following is a key lemma\footnote{We prove the lemma in the same manner as \cite[Theorem 1.5]{E}, but our assertion seems to be a weaker than \cite[Theorem 1.5]{E}. We found gaps in the proof of \cite[Theorem 1.5]{E}, and the current lemma is the best of what we can prove avoiding the found gaps.}.

\begin{lemma}\label{lem:appro1}
Consider a definably complete locally o-minimal expansion of an ordered field $\mathcal F=(F,<,+,0,\cdot,1,\ldots)$.
Let $r$ be a nonnegative integer and $(X,\pi,T,\eta,\rho)$ be a special $\mathcal C^{r+1}$ submanifold in $F^n$ of dimension $d$ with a tubular neighborhood, where $\pi$ is the coordinate projection onto the first $d$-coordinates.
Set $U=\pi(X)$.
Let $f:T \rightarrow F$ be a $\mathcal D^r$ function which is of class $\mathcal C^{r+1}$ off $X$.
Assume that the restriction $f|_X$ of $f$ to $X$ is of class $\mathcal C^{r+1}$.
Let $\delta:T \rightarrow F$ be a positive definable continuous function.
Then there exists a $\mathcal D^{r+1}$ function $\widetilde{f}:T \rightarrow F$ such that $$|f-\widetilde{f}|<\delta$$ on $T$ and $$ \left|\dfrac{\partial^{|a|+|b|}}{\partial x^a \partial y^b}(f(x,y)-\widetilde{f}(x,y))\right| \to 0$$ as $(x,y)$ approaches to $X$ for all sequences of nonnegative integers $a$ and $b$ with $0<|a| + |b| \leq r$.
Furthermore, we may assume that $\widetilde{f}$ coincides with $f$ outside a small neighborhood of $X$ in $T$.
\end{lemma}
\begin{proof}
There is noting to prove when $d=n$.
We assume that $d<n$.
By Taylor's theorem, we have 
\begin{align*}
f(x,y) &= \displaystyle\sum_{|b|<r}\dfrac{1}{b!}\dfrac{\partial^{|b|}f(x,\rho(x,y))}{\partial y^b}(y-\rho(x,y))^b\\
&\qquad+\sum_{|b|=r}\dfrac{1}{b!}\dfrac{\partial^{|b|}f(x,\xi y+(1-\xi)\rho(x,y))}{\partial y^b}(y-\rho(x,y))^b
\end{align*}
for a suitable $0<\xi<1$ when $(x,y) \in T$.
Set $$P(x,y)=\displaystyle\sum_{|b| \leq r}\dfrac{1}{b!}\dfrac{\partial^{|b|}f(x,\rho(x,y))}{\partial y^b}(y-\rho(x,y))^b.$$
Recall that the $\mathcal D^{r+1}$ retraction $\rho(x,y)$ is locally constant with respect to the variables $y$ when $x$ is fixed. 
We assumed that the restriction of $f$ to $X$ is of class $\mathcal C^{r+1}$.
Therefore, the function $P(x,y)$ is a $\mathcal D^{r+1}$ function.
We may assume that $\eta(x) \to 0$ when $x$ approaches to a point in the boundary of $U$.
In fact, there is a definable nonnegative continuous function $d:F^n \to F$ whose zero set is the boundary $\partial U$ of $U$ because $\partial U$ is closed.
We have only consider $\min\{\eta,d\}$ instead of $\eta$.
Take a sufficiently small positive $\mathcal D^{r+1}$ function $\varepsilon: U \rightarrow F$ so that $\varepsilon(x)<\eta(x)$.
Note that $\varepsilon(x) \to 0$ as $x$ approaches to the boundary of $U$.
We can take such a function thanks to Lemma \ref{lem:aaa}.
Let $\mu:F \rightarrow F$ be a $\mathcal D^{r+1}$ function such that it is constantly one in a neighborhood of $0$ and constantly zero outside the interval $[-1,1]$. 
Set $$T_{\varepsilon}=\{(x,y) \in T\;|\; \|y-\rho(x,y)\|<\varepsilon(x)\}$$ and $$\lambda(x,y)=\mu\left(\dfrac{\|y-\rho(x,y)\|^2}{\varepsilon(x)^2}\right).$$
Consider the function $\widetilde{f}:T \rightarrow F$ given by
$$\widetilde{f}(x,y)=\lambda(x,y)P(x,y)+(1-\lambda(x,y))f(x,y).$$
It is well-defined because $T_{\varepsilon} \subseteq T$ and $\widetilde{f}(x,y)=f(x,y)$ outside of $T_{\varepsilon}$.
Since $\lambda$ is one in a neighborhood of $X$, we have $\widetilde{f}(x,y)=P(x,y)$ in a neighborhood of $X$.
Therefore, the function $\widetilde{f}$ is of class $\mathcal C^{r+1}$ everywhere in $T$.  
Since $\widetilde{f}=f$ outside of $T_{\varepsilon}$, we have only to demonstrate that $\widetilde{f}$ is an approximation of $f$ in $T_{\varepsilon}$.

We can choose $\varepsilon$ so that the absolute values of $|f-\widetilde{f}|$ is smaller than $\delta$.
We want to show this fact.
For that purpose, we need the following claim:
\medskip

\textbf{Claim.} Let $R:T \to F$ be a definable continuous map such that $R(x,y)=0$ for every $(x,y) \in X$.
Then there exists a $\mathcal D^r$ function $\varepsilon: U \to F$ such that $|R(x,y)|<\delta(x,y)$ whenever $x \in U$ and $|y-\rho(x,y)|<\varepsilon(x)$.
\begin{proof}[Proof of Claim]
Fix $x \in U$.
The definable set $X_x:=X \cap \pi^{-1}(x)$ is discrete and closed.
We consider the map $\sigma:X_x \times (0,1)\to F$ defined by 
\begin{align*}
\sigma(y,c) &:=\sup\{0<r<\varepsilon(x)/2\;|\; \forall y \ ((x,y) \in T_{\varepsilon} \text{ and }\|y-\rho(x,y)\| <r) \\
& \qquad \rightarrow (|R(x,y)|<c \cdot \delta(x,y))\}.
\end{align*}
Observe that the inequality $|R(x,y)|<c \cdot \delta(x,y)$ is equivalent to the inequality $|R(x,y)|/\delta(x,y)<c$ because $\delta>0$. 
We have $\sigma(y,c)>0$ because $R(x,\rho(x,y))=0$ and $R/\delta$ is continuous.
Consider the map $\tau:U \times (0,1) \to F$ given by $\tau(x,c)=\inf\{\sigma(y,c)\;|\; y \in X_x\}$.
We have $\tau(x,c)>0$ for every $x \in U$ and $0<c<1$ because the set $\{\sigma(y,c)\;|\; y \in X_x\}$ is closed and discrete by Proposition \ref{prop:pre}(1),(5).
We show that, for each $x_0 \in U$, there exists an open neighborhood $V_0$ such that $\tau(x,1/2) \geq \frac{1}{2}\tau(x_0,1/4)$ for each $x \in V_0$.

Set $W=\{(x,y) \in T\;|\; |R(x,y)| \geq \frac{1}{2} \delta(x,y)\} \cup (F^n \setminus T)$.
Observe that $W$ is a definable closed subset of $F^n$ because $R$ and $\delta$ are continuous, and $T$ is open.
We set $d_W(z):=\inf\{\|z-z'\|\;|\; z' \in W\}$ for each $z \in F^n$.
The definable function $d_W$ is continuous.
We define $\sigma':X_{x_0} \to F$ by $\sigma'(y)=\inf\{d_W(x_0,y')\;|\; \|y' - \rho(x_0,y)\| \leq \frac{1}{2}\tau(x_0,1/4)\}$.
We have $\sigma'(y)>0$ by \cite[Corollary(Max-min theorem)]{M} because the definable set $\{y' \in F^{n-m}\;|\; \|y' - \rho(x_0,y)\| \leq \frac{1}{2}\tau(x_0,1/4)\}$ is a definable closed and bounded set which has an empty intersection with $W$ and $d_W$ is continuous.
We consider $w:=\inf\{\sigma'(y)\;|\; y \in X_{x_0}\}$.
The value $w$ is positive for the same reason as $\tau(x,c)$ is positive.
Observe that the definable set $A:=\{(x,y) \in F^m \times F^{n-m}\;|\; \|x-x_0\| < w, \|y -\rho(x,y)\| \leq \frac{1}{2}\tau(x_0,1/4)\}$ has an empty intersection with $W$.
It means that the inequality $|R(x,y)|<\frac{1}{2}\eta(x,y)$ holds for each $(x,y) \in A$.
Set $V_0=\{x \in U\;|\; \|x-x_0\| <w\}$.
We have shown that $\tau(x,1/2) \geq \frac{1}{2}\tau(x_0,1/4)$ for each $x \in V_0$.
In particular, the function $\tau(x,1/2)$ is locally bounded from below by positive constants.
We can choose a $\mathcal D^{r+1}$ function $\varepsilon':U \to F$ such that $\varepsilon'(x)<\tau(x,1/2)$ by Lemma \ref{lem:aaa}.
We may assume that the inequality $|R(x,y)|<\delta(x,y)$ holds whenever $x \in U$ and $|y-\rho(x,y)|<\varepsilon(x)$.
\end{proof}

Set $Q(x,y)=f(x,y) - P(x,y)$.
The definable function $Q$ is of class $\mathcal D^r$, it is of class $\mathcal D^{r+1}$ off $X$ and its restriction to $X$ is of class $\mathcal D^{r+1}$  because $f$ is so and $P$ is of class $\mathcal D^{r+1}$. 
We can choose $\varepsilon$ so that $|Q|<\delta$ on $T_{\varepsilon}$ by Claim.
We have $|\widetilde{f}-f|=\lambda \cdot |Q| \leq |Q|<\delta$ on $T_{\varepsilon}$ because $0 \leq \delta \leq 1$.
We have $|\widetilde{f}-f|<\delta$ because $\widetilde{f}=f$ off $T_{\varepsilon}$.

We have to check that the derivatives of $f-\widetilde{f}$ is sufficiently small.
We have 
$$\dfrac{\partial^{|a|+|b|}(f-\widetilde{f})}{\partial x^{a} \partial y^{q}}=\displaystyle\sum_{p \leq a, q \leq b} \mathcal A_{p,q} \dfrac{\partial^{|p|+|q|}\lambda}{\partial x^{p} \partial y^{q}}\dfrac{\partial^{|a-p|+|b-q|}Q}{\partial x^{a-p} \partial y^{b-q}},$$
where $\mathcal A_{p,q}$ are constants.
Observe that $\dfrac{\partial^{|a-p|+|b-q|}Q}{\partial x^{a-p} \partial y^{b-q}}(x,y)=0$ if $(x,y) \in X$.
Note that $$\frac{\partial \rho(x,y)}{\partial y}=0$$ because $\rho(x,y)$ is locally constant with respect to $y$.
Observe that the derivatives of $\mu$ is bounded because they are locally constant outside of the definably compact set $[-1,1]$.
Using the above facts, we can find the upper bounds of $\dfrac{\partial^{|p|+|q|}\lambda}{\partial x^{p} \partial y^{q}}$ in the same manner as \cite[p.116-p.117]{E}, but it is technical and we omit the details.
Combining them, we can prove that $\dfrac{\partial^{|a|+|b|}(f-\widetilde{f})}{\partial x^{a} \partial y^{q}}(x,y) \to 0$ when $(x,y) \in T$ and it approaches to the boundary of $X$.
\end{proof}

Once Lemma \ref{lem:appro1} is proved, we can prove the following theorem in the same manner as \cite[Theorem 1.6]{E}.
\begin{theorem}\label{thm:appro}
Consider a definably complete locally o-minimal expansion of an ordered field $\mathcal F=(F,<,+,0,\cdot,1,\ldots)$.
Let $r$ be a nonnegative integer and $\delta:U \rightarrow F$ be a positive continuous definable function defined on a definable open subset $U$ of $F^n$.
Let $f:U \to F$ be a $\mathcal D^r$ function.
There exists a $\mathcal D^{r+1}$ approximation $\widetilde{f}:U \rightarrow F$ of $f$ such that $$|f-\widetilde{f}|<\delta.$$
\end{theorem}
\begin{proof}
The counterpart of this theorem for o-minimal structures is \cite[Theorem 1.6]{E}.
In its proof, a stratification of the ambient space $F^n$ into cells is used.
It is unavailable in our setting.
Instead, we use a decomposition of $F^n$ into special $\mathcal C^{r+1}$ submanifolds with tubular neighborhoods satisfying the frontier condition given in Theorem \ref{thm:tubular_decom}. 
We can complete the proof using Lemma \ref{lem:appro1} in place of \cite[Theorem 1.5]{E} in the same manner as \cite[Theorem 1.6]{E}.
We give a proof here for readers' convenience.
\medskip

\textbf{Claim 1.}
There exists a decomposition $\{(X_i,\pi_i,T_i,\eta_i,\rho_i)\}_{i=1}^N$ of $F^n$ into special $\mathcal C^{r+1}$ submanifolds with tubular neighborhoods and satisfying the frontier condition such that the restriction of $f$ to $X_i$ is of class $\mathcal D^{r+1}$ whenever $X_i \subseteq U$.
\begin{proof}
	We show that, for every $0 \leq k \leq n$, there exists a decomposition $$\{(X_i,\pi_i,T_i,\eta_i,\rho_i)\}_{i=1}^N$$ of $F^n$ into special $\mathcal C^{r+1}$ submanifolds with tubular neighborhoods partitioning $\{U\}$  and satisfying the frontier condition such that the restriction of $f$ to $X_i$ is of class $\mathcal D^{r+1}$ whenever $\dim X_i \geq n-k$ and $X_i \subseteq U$.
	We prove it by induction on $k$.
	
	We first consider the case in which $k=0$.
	Let $D_0$ be the set of points at which $f$ is not of class $\mathcal D^{r+1}$.
	We have $\dim D_0<n$ by Proposition \ref{prop:cr_pre}.
	A decomposition of $F^n$ into special $\mathcal C^{r+1}$ submanifolds with tubular neighborhoods partitioning $\{D_0, U\}$ and  satisfying the frontier condition meets the requirement.
	
	We next consider the case $k>0$.
	By the induction hypothesis, there exists a decomposition $\{(X'_i,\pi'_i,T'_i,\eta'_i,\rho'_i)\}_{i=1}^{N'}$ of $F^n$ into special $\mathcal C^{r+1}$ submanifolds with tubular neighborhoods partitioning $\{U\}$ and satisfying the frontier condition such that the restriction of $f$ to $X'_i$ is of class $\mathcal D^{r+1}$ whenever $\dim X'_i > n-k$ and $X'_i \subseteq U$.
	We may assume that $X'_i \subseteq U$ if and only if $1 \leq i \leq p$, $\dim X'_i > n-k$ for $1 \leq i \leq l$, $\dim X'_i = n-k$ for $l<i \leq m$ and $\dim X'_i < n-k$ for $l<i \leq p$.
	Let $D_i \subseteq X'_i$ be the set of points at which the restriction of $f$ to $X'_i$ is not of  class $\mathcal D^{r+1}$ $l<i \leq m$.
	
	We can get a decomposition $\{(X''_i,\pi''_i,T''_i,\eta''_i,\rho''_i)\}_{i=1}^{N''}$  of $F^n$ into special $\mathcal C^{r+1}$ submanifolds with tubular neighborhoods partitioning $\{D_i,X'_i\}_{l<i \leq m} \cup \{X'_i\}_{i>m}$ and  satisfying the frontier condition.
	We assume that $X''_i \cap X'_j=\emptyset$ for every $1 \leq j \leq l$ if and only if $1 \leq i \leq l'$.
	The decomposition $\{(X'_i,\pi'_i,T'_i,\eta'_i,\rho'_i)\}_{1 \leq i \leq l} \cup \{(X''_i,\pi''_i,T''_i,\eta''_i,\rho''_i)\}_{i=1}^{l'}$ is a desired decomposition.
\end{proof}

Let $\{(X_i,\pi_i,T_i,\eta_i,\rho_i)\}_{i=1}^N$ be the decomposition given in Claim 1.
We next show the following claim:
\medskip

\textbf{Claim 2.} We may assume that $T_i \cap U=\emptyset$ if $X_i \subseteq U$.
\begin{proof}[Proof of Claim 2.]
	Consider the definable continuous map $\mathfrak c_i:X_i \to F$ given by $\mathfrak c_i(x)=\inf\{|x-y|\;|\; y \in F^n \setminus U\}$.
	It is a positive definable continuous function.
	Consider the definable function $\mathfrak  d_i:\pi_i(X_i) \to F$ given by $\mathfrak d_i(t)=\inf \{\mathfrak c_i(x)\;|\; x \in X \cap \pi^{-1}(t)\}$.
	This is positive and continuous by Proposition \ref{prop:dist}.
	Replace $\eta_i$ with $\min\{\eta_i,\mathfrak d_i\}$.
	Then, $T_i \cap U=\emptyset$ by the definition of special manifolds with tubular neighborhoods. 
\end{proof}

Let $D_k$ be the union of $X_i$'s satisfying the inequality $\dim X_i < n-k$ and the inclusion $X_i \subset U_i$.
Observe that $D_k$ is closed in $U$ by the frontier condition. 
We prove the following claim by induction on $k$.
The theorem immediately follows from the claim.
\medskip

\textbf{Claim 3.}
Let $0 \leq k \leq n$.
There exist a nonnegative definable continuous function $\eta_{k,\alpha}:F^n \to F$ and a $\mathcal D^{r+1}$ approximation $\widetilde{f}_k$ of the restriction of $f$ to $F^n \setminus D_k$ such that $\eta_{k}<\delta$, $\eta_{k} \equiv 0$ on $D_k$ and $\left|\dfrac{\partial^{|\alpha|}}{\partial x^{\alpha}}(f-\widetilde{f}_k)\right| \to 0$ as $(x,y)$ approaches to a point in $D_k$ for $0 \leq |\alpha| \leq r$.
\begin{proof}[Proof of Claim 3]
	We prove the claim by induction on $k$.
	We first consider the case in which $k=0$.
	Let $\widetilde{f}_0$ be the restriction of $f$ to $U \setminus D_0$.
	We define $\eta_{0}(x)=0$.
	
	We next consider the case in which $k>0$.
	There exist a nonnegative definable continuous function $\eta_{k-1,}:U \to F$ and a $\mathcal D^{r+1}$ approximation $\widetilde{f}_{k-1}$ of the restriction of $f$ to $U \setminus D_{k-1}$ such that $\eta_{k-1}<\delta$, $\eta_{k-1,\alpha} \equiv 0$ on $D_{k-1}$, 
	\begin{equation}
	|f-\widetilde{f}_{k-1}|< \eta_{k-1}\label{eq:aaa}
	\end{equation}
	on $U \setminus D_{k-1}$ and 
	\begin{equation}
	\left|\dfrac{\partial^{|\alpha|}}{\partial x^{\alpha}}(f-\widetilde{f}_{k-1})(x)\right| \to 0\label{eq:ccc}
	\end{equation}
	as $x$ approaches to a point in $D_{k-1}$ for $0 < |\alpha| \leq r$.
	 Take a nonnegative definable continuous function $\nu:U \to F$ whose zero set is $D_{k}$.
	 Set $\eta_{k}(x)=\max\{\eta_{k-1}(x),\frac{\nu(x)}{\nu(x)+1}\delta(x)\}$ for $x \in U$.
	 The definable function $\eta_k$ is nonnegative, continuous and satisfies the conditions $\eta_{k}<\delta$ and $\eta_{k} \equiv 0$ on $D_{k}$.
	
	Fix a special $\mathcal C^r$ submanifolds $(X,\pi,T,\eta,\rho)$ of dimension $n-k$ with tubular neighborhood such that it is a member of the decomposition $\{(X_i,\pi_i,T_i,\eta_i,\rho_i)\}_{i=1}^N$ and it is contained in $D_{k-1}$.
	We have $T \subseteq U$ by Claim 2.
	Shrinking $T$ if necessary, we may assume that $X_i \cap T=\emptyset$ for every $X_i$ of dimension $n-k$ with $X_i \neq X$.
	The restriction of $f$ to $X$ is of class $\mathcal D^{r+1}$ by Claim 1.  
	Consider the restriction of $\widetilde{f}_{k-1}$ to $T \setminus D_{k-1}$ and the restriction of $f$ to $X$.
	They are of class $\mathcal D^{r+1}$.
	Consider the definable function $g_X:T \to F$ given by $g_X(x)=\widetilde{f}_{k-1}(x)$ for $x \in T \setminus D_{k-1}$ and $g_X(x)=f(x)$ elsewhere.
	The conditions (\ref{eq:aaa}) and (\ref{eq:ccc}) imply that $g_X$ is of class $\mathcal D^r$.
	We have succeeded in constructing a $\mathcal D^r$ function $g_X:T \to F$ whose restriction to $X$ is of class $\mathcal D^{r+1}$.
	Apply Lemma \ref{lem:appro1} to $g_X$ and $\eta_{k}$.
	We can find a $\mathcal D^{r+1}$ function $\widetilde{g}_X:T \rightarrow F$ such that $$|\widetilde{g}_X -g |<\delta_k$$ on $T$ and $$\dfrac{\partial^{|a|}}{\partial x^a}(g_X-\widetilde{g}_X)(x) \to 0$$ as $x$ approaches to $X$ for all sequences of nonnegative integers $a$ with $0<|a|  \leq r$.
	Furthermore, $\widetilde{g}_X$ coincides with $\widetilde{f}_{k-1}$ outside of a small neighborhood of $X$ in $T$.
	Consider the definable function $\widetilde{f}_X:(U \setminus D_{k-1}) \cup X \to F$ defined by $\widetilde{g}_X(x)$ when $x \in T$ and $\widetilde{f}_{k-1}(x)$ elsewhere.
	It is a $\mathcal D^{r+1}$ approximation of the restriction of $f$ to $(U \setminus D_{k-1}) \cup X$.
	
	We have finitely many members $(X,\pi,T,\eta,\rho)$ of dimension $n-k$ with tubular neighborhood such belonging to $\{(X_i,\pi_i,T_i,\eta_i,\rho_i)\}_{i=1}^N$.
	We extended the domain of definition of an approximation from $U \setminus D_{k-1}$ to $(U \setminus D_{k-1}) \cup X$.
	Applying the above procedure finitely many times and extend the domain of definition of the approximation step by step, we can construct a definable approximation $\widetilde{f}_k$ of $f$ restricted to $U \setminus D_k$ satisfying the requirements.
\end{proof}
\end{proof}

Repeating the proofs of \cite{E} using Theorem \ref{thm:tub} and Theorem \ref{thm:appro}, we get the following results:
\begin{theorem}\label{thm:appro_nbd}
Consider a definably complete locally o-minimal expansion of an ordered field $\mathcal F=(F,<,+,0,\cdot,1,\ldots)$ and a $\mathcal D^r$ submanifold $X$ of $F^n$ for some positive integer $r$.
Then there exists a tubular $\mathcal D^r$ neighborhood of $X$; that is, there exist a definable open neighborhood $U$ of $X$ in $F^n$ and a $\mathcal D^r$ retraction $\rho:U \rightarrow X$. 
\end{theorem}
\begin{proof}
The literally same proof as \cite[Theorem 1.8, Theorem 1.9]{E} demonstrates the theorem.
\end{proof}

We can get the following theorem:
\begin{theorem}\label{thm:appro2}
	Consider a definably complete locally o-minimal expansion of an ordered field $\mathcal F=(F,<,+,0,\cdot,1,\ldots)$ and two $\mathcal D^r$ submanifolds $X$ and $Y$.
	A $\mathcal D^{r-1}$ map $f:X \rightarrow Y$ admits a $\mathcal D^r$ approximation; that is, for any positive definable continuous function $\varepsilon$ on $X$, there exists a definable $\mathcal D^{r}$ map $\widetilde{f}:X \rightarrow Y$ such that $\|f-\widetilde{f}\|<\varepsilon$. 
\end{theorem}
\begin{proof}
	The proof is similar to the proof of \cite[Theorem 1.1]{E}.
	We can first reduce to the case in which $X$ is a definable open subset in the same manner as \cite[Theorem 1.8]{E}.
	We next construct a $\mathcal D^r$ retraction $\rho:U \to Y$ by Theorem \ref{thm:appro_nbd}.
	We next consider the definable map $\mu:X \to F$ defined by $\mu(x)=\sup\{r>0\;|\; \forall y \in F^n\ \|y-f(x)\|<r \rightarrow y \in U \wedge \|\rho(y)-f(x)\|<\varepsilon(x)\}$, where $F^n$ is the ambient space of $Y$.
	Observe that $\mu$ is bounded below by a positive constant and apply a definable continuous function $\delta:X \to F$ such that $\delta<\mu$.
	We then find a $\mathcal D^r$ approximation $\widetilde{f}:X \to F^n$ by Theorem \ref{thm:appro}.
	The composition $\rho \circ \widetilde{f}$ is a desired approximation.
\end{proof}

\begin{remark}
In \cite{E}, the triviality theorems for families definable in an o-minimal expansion of an ordered field are demonstrated as an application of the approximation theorem.
However, the triviality fails for a definably complete locally o-minimal expansion of an ordered field.

We first recall definitions.
Fix a definably complete locally o-minimal expansion of an ordered field $\mathcal F=(F,<,+,0,\cdot,1,\ldots)$.
Let $A$ and $S$ be definable subsets of $F^m$ and $F^n$, respectively.
A \textit{definable trivialization} of a definable map $f:S \rightarrow A$ is a pair $(Z,\lambda)$ of a definable subset $Z$ of $F^N$ for some $N$ and a definable map $\lambda: S \rightarrow Z$ such that the map $(f,\lambda):S \rightarrow A \times Z$ is a homeomorphism.
The map $f$ is \textit{definably trivial} if it has a definable trivialization, and $f$ is \textit{definably trivial over $A'$} for a definable subset $A'$ of $A$, the restriction $f|_{f^{-1}(A')}:f^{-1}(A') \rightarrow A'$ is definably trivial.
A weak triviality  asserts that there exists a partition $A=A_1 \cup \cdots \cup A_k$ into definable sets such that $f$ is definably trivial over $A_i$ for each $1 \leq i \leq k$.
It is satisfied for o-minimal expansions of ordered fields \cite[Chapter 9, Theorem 1.2]{vdD}.

We construct a counterexample to the weak triviality when the structure is a definably complete locally o-minimal expansion of an ordered field.
The triviality theorems given in \cite{E} are stronger than the weak triviality.
The example constructed here is also a counterexample to them. 

Let $\mathcal M=(M,<,+,0,\cdot,1,\ldots)$ be a fixed o-minimal structure in a language $\mathcal L$.
We may consider that the set of integers $\mathbb Z$ is a subset of $M$.
We consider a binary predicate $P$.
Set $$P_n=\{(i,j) \in \mathbb Z^2\;|\; 1 \leq i \leq n,\ 1 \leq j \leq i\}$$ for all positive integers $n$.
The structure $\mathcal M_n=\langle \mathcal M, P_n\rangle$ is an $\mathcal L(P)$-expansion of $\mathcal M$, where $P$ is interpreted by $P_n$.
They are also o-minimal.
In particular, they are locally o-minimal and definably complete.
Take a non-principal ultrafilter $I$ of the set of positive integers $\mathbb N$.
The ultraproduct $\mathcal F=\prod_{n \in \mathbb N} \mathcal M_n/I=(F,<,+,0,\cdot,1,P,\ldots)$ is an elementary extension of $\mathcal M$ by \L o\'{s}'s theorem.
In particular, it is a definably complete locally o-minimal expansion of an ordered field because local o-minimality and definable completeness are expressed by first-order sentences.
We demonstrate that the weak triviality fails in $\mathcal F$.

Consider the definable map $f:P=\{(x,y) \in F^2\;|\; \mathcal F \models P(x,y)\} \rightarrow F$ given by $f(x,y)=x$.
Set $a_n=[(n,n,\ldots)] \in N$ for all $n \in \mathbb N$.
Here, the notation $[(n,n,\ldots)]$ denotes the equivalence class of $(n,n,\ldots) \in \prod_{n \in \mathbb N}M$ in $F=(\prod_{n \in \mathbb N}M)/I$.
We show that the cardinality of $f^{-1}(a_n)$ is $n$.
If the weak triviality holds true, the cardinalities of two inverse images $f^{-1}(a)$ and $f^{-1}(a')$ coincide when both the points $a$ and $a'$ is contained in a single $A_l$ for some $1 \leq l \leq k$.
So, the claim implies that the map $f$ is a counterexample to the weak triviality.

Take $b=[(b_1,b_2,\ldots)]$ so that $(a_n,b) \in f^{-1}(a_n)$.
Consider the partition $$\mathbb N= \bigcup_{i=1}^n \{ j\in \mathbb N\;|\; b_j = i\} \cup \{j \in \mathbb N\;|\; b_j \neq i \text{ for all } i=1,\ldots, n\}.$$
One of the sets $\{ j\in \mathbb N\;|\; b_j = i\}$ and $\{j \in \mathbb N\;|\; b_j \neq i \text{ for all } i=1,\ldots, n\}$ is contained in the ultrafilter $I$.
Since $(a,b) \in P$, the last set is not contained in $I$.
So, we can find $1 \leq i \leq n$ such that $\{ j\in \mathbb N\;|\; b_j = i\} \in I$.
We get $[(b_1,b_2,\ldots)]=[(i,i,\ldots)]$ by the definition of $\mathcal F$.
We have demonstrated that $f^{-1}(a_n)=\{(a_n,b) \in F^2\;|\; b = [(i,i,\ldots)] \text{ for some }  i \in \mathbb N \text{ with } 1 \leq i \leq n\}$.
\end{remark}

\subsection{Definable Positivstellensatz}\label{sec:positivstellensatz}
Positivstellensatz is a main achievement of real algebraic geometry.
Readers who are interested in it should consult \cite{BCR}.
Positivstellensatz for $\mathcal C^r$ functions definable in an o-minimal expansion of an ordered field is also demonstrated in \cite{AAB}.
We derive Positivstellensatz for $\mathcal D^r$ functions when the structure is a definably complete locally o-minimal expansion of an ordered field.
We first show the following lemmas:

\begin{lemma}\label{lem:vddm}
	Let $\mathcal F=(F,<,+,0,\cdot,1,\ldots)$ be a definably complete locally o-minimal expansion of an ordered field.
	The following assertions hold true:
	\begin{enumerate}
		\item[(1)] Let $g:A \times F \rightarrow F$ be a definable function with $A \subseteq F^m$.
		Then there exist definable functions $\psi:F \rightarrow F$ and $\rho:A \rightarrow F$ such that $|g(x,t)| < \psi(t)$ for all $x \in A$ and $t>\rho(x)$.
		\item[(2)] Let $f,g:F^n \rightarrow F$ be definable continuous functions of class $\mathcal C^r$ on $F^n \setminus g^{-1}(0)$ with $f^{-1}(0) \subseteq g^{-1}(0)$.
		Then there exist an odd increasing $\mathcal D^r$ bijection $\phi:F \rightarrow F$ and a $\mathcal D^r$ function $h:F^n \rightarrow F$ such that $\phi$ is $r$-flat at $0$ and $\phi \circ g=hf$.
		\item[(3)] Let $f:F^n \rightarrow F$ be a definable continuous functions of class $\mathcal C^r$ on $F^n \setminus f^{-1}(0)$.
		Then there exists an odd increasing $\mathcal D^r$ bijection $\phi:F \rightarrow F$ such that $\phi$ is $r$-flat at $0$ and $\phi \circ f$ is a $\mathcal D^r$ function.
	\end{enumerate}
\end{lemma}
\begin{proof}
	(1) The counterpart of this assertion in o-minimal setting is \cite[Proposition C.4]{vdDM}.
	The original proof is almost applicable to our settings using Proposition \ref{prop:pre}(8),(9) instead of monotonicity theorem for o-minimal structures.
	In the original proof, the cell decomposition theorem is used to reduce to the case in which a definable function $\xi$ defined on $A$ is continuous.
	The cell decomposition theorem is not available in our setting.
	Instead, we prove the assertion by induction on $\dim A$ using Proposition \ref{prop:cr_pre}.
	We omit the details.
	
	(2) The counterpart is \cite[Proposition C.9]{vdDM}.
	l'Hopital's rule is used in the original proof.
	It follows from the intermediate value theorem according to a classical proof.
	Since the intermediate value theorem holds true for continuous functions definable in a definably complete structure \cite{M}, l'Hopital's rule also holds true for definable functions.
	The literally same proof as the original proof demonstrates our assertion by using the assertion (1) in place of \cite[Proposition C.4]{vdDM}.
	
	(3) It is a corollary of (2).
\end{proof}

\begin{lemma}\label{lem:vdDM2}
Let $\mathcal F=(F,<,+,0,\cdot,1,\ldots)$ be a definably complete locally o-minimal expansion of an ordered field.
Let $A \subseteq F^m$ be a definable set and $f:A \times (0,\infty) \rightarrow F$ be a definable function.
There exists an odd increasing $\mathcal D^r$ bijection $\varphi:F \rightarrow F$ which is $r$-flat at $0$ such that $\displaystyle\lim_{t \to 0+}\varphi(t)f(x,t)=0$ for each $x \in A$. 
\end{lemma}
\begin{proof}
The assertion \cite[Lemma C.7]{vdDM} is almost the same as the lemma, but it is an assertion for o-minimal structures.
In the original proof, they used definable cell decomposition which is unavailable in our situation.
However, we can prove the lemma in the same manner as the original one using Proposition \ref{prop:pre}(9) in place of definable cell decomposition.
\end{proof}

The following Positivstellensatz holds true:
\begin{theorem}[Definable Positivstellensatz]\label{thm:positivstellensatz}
Let $\mathcal F=(F,<,+,0,\cdot,1,\ldots)$ be a definably complete locally o-minimal expansion of an ordered field.
Let $f_1, \ldots, f_k$ be $\mathcal D^r$ functions on $F^n$ such that the set $$S=\{x \in F^n\;|\; f_i(x) \geq 0\ \ (i=1,\ldots,k)\}$$
is not empty.
Let $g$ be a $\mathcal D^r$ function on $F^n$.
The following assertions hold true:
\begin{enumerate}
\item[(i)] If $g \geq 0$ on $S$, there exist $\mathcal D^r$ functions $p, v_0,\ldots, v_k$ on $F^n$ such that $p^{-1}(0) \subseteq g^{-1}(0)$ and $p^2g=v_0^2+\sum_{i=1}^k v_i^2f_i$.
\item[(ii)]  If $g > 0$ on $S$, there exist $\mathcal D^r$ functions $v_0,\ldots, v_k$ on $F^n$ such that $g=v_0^2+\sum_{i=1}^k v_i^2f_i$.
\end{enumerate}
\end{theorem}
\begin{proof}
The proof of \cite[Theorem 3.7]{AAB} works also in our setting.
We have to use Lemma \ref{lem:vdDM2}, Lemma \ref{lem:vddm}(3) and Theorem \ref{thm:appro} instead of o-minimal counterparts.
\end{proof}

\section{Imbedding of definable $\mathcal C^r$ manifolds}\label{sec:manidfolds}
\subsection{Imbedding theorem}

We first define definable $\mathcal C^r$ manifolds and definable $\mathcal C^r$ maps between them when the structure is an definably complete locally o-minimal expansion of an ordered field.
The definition of a definable $\mathcal C^r$ manifold is tricky and different from that in the o-minimal setting \cite{K}, but it is useful in our settings.

\begin{definition}
	Let $\mathcal F=(F,<,+, \cdot, 0,1,\ldots)$ be a definably complete locally o-minimal expansion of an ordered field.
	Suppose that $1 \le r < \infty$.
	
%	(1) A definable subset $X$ of $F^n$ is called a \textit{$d$-dimensional definable $\mathcal C^r$ submanifold of $F^n$}
%	We call a definable $\mathcal C^r$ submanifold a \textit{$\mathcal D^r$ submanifold} for short.
%	if for any $x \in X$ there exists a $\mathcal D^r$ diffeomorphism
%	$\phi_x$ from some open box $U_x$ of the origin in $F^n$ 
%	onto some open box $V_x$ of $x$ in $F^n$ 
%	such that $\phi_x (0)=0$ and $\phi(M^d \cap U_x) = X \cap V_x$. 
%	Here, $F^d$ is regarded as the subset of $F^n$ whose last $(n-d)$ coordinates are zero.
%	
	(1)
	A pair $(M, \{\varphi_i:U_i \rightarrow U'_i\}_{i \in I})$ of a topological space and a finite family of homeomorphisms is called a \textit{definable $\mathcal C^r$ manifold} or a \textit{$\mathcal D^r$ manifold} if  
	\begin{itemize}
		\item $\{U_i\}_{i \in I}$ is a finite open cover of $M$, 
		\item $U'_i$ is a $\mathcal D^r$ submanifold of $F^{m_i}$ for any $i \in I$ and,
		\item the composition $(\varphi_j|_{U_i \cap U_j}) \circ (\varphi_i|_{U_i \cap U_j})^{-1}:\varphi_i(U_i \cap U_j) \rightarrow \varphi_j(U_i \cap U_j)$ is a $\mathcal D^r$ diffeomorphism whenever $U_i \cap U_j \neq \emptyset$.
	\end{itemize}
	Here, the notation $\varphi_i|_{U_i \cap U_j}$ denotes the restriction of $\varphi_i$ to ${U_i \cap U_j}$.
	We use similar notations throughout the rest of this paper.
	The family $\{\varphi_i:U_i \rightarrow U_i'\}_{i \in I}$ is called a \textit{ $\mathcal D^r$ atlas} on $M$.
	We often write $M$ instead of $(M, \{\varphi_i:U_i \rightarrow U'_i\}_{i \in I})$ for short.
	Note that a $\mathcal D^r$ submanifold is naturally a $\mathcal D^r$ manifold.

	In the o-minimal setting, a $\mathcal D^r$ manifold is defined as the object obtained by pasting finitely many definable open sets.
	$\mathcal D^r$ submanifolds are pasted in our definition. 
	If we adopt the same definition of $\mathcal D^r$ manifolds as in the o-minimal setting, $\mathcal D^r$ manifolds of dimension zero should be a finite set because $F^0$ is a singleton.
	A $\mathcal D^r$ submanifold of dimension zero is not necessarily a $\mathcal D^r$ manifold in this definition.
	It seems to be strange, so we employed our definition of $\mathcal D^r$ manifolds.
	
	Given a $\mathcal D^r$ manifold $M$, two $\mathcal D^r$ atlases $\{\varphi_i:U_i \rightarrow U_i'\}_{i \in I}$ and $\{\psi_j:V_j \rightarrow V_j'\}_{j \in J}$ on $M$ are \textit{equivalent} if, for all $i \in I$ and $j \in J$,
	\begin{itemize}
		\item the images $\varphi_i(U_i \cap V_j)$ and $\psi_j(U_i \cap V_j)$ are open definable subsets of $U'_i$ and $V'_j$, respectively, and
		\item the $\mathcal D^r$ diffeomorphism $(\psi_j|_{U_i \cap V_j}) \circ (\varphi_j|_{U_i \cap V_j})^{-1}:\varphi_i(U_i \cap V_j) \rightarrow \psi_j(U_i \cap V_j)$ are definable whenever $U_i \cap U_j \neq \emptyset$.
	\end{itemize}
	The above relation is obviously an equivalence relation.
	
	A subset $X$ of the $\mathcal D^r$ manifold $M$ is \textit{definable} when $\varphi_i(X \cap U_i)$ are definable for all 
	$i \in I$. 
	When two atlases $\{\varphi_i:U_i \rightarrow U_i'\}_{i \in I}$ and $\{\psi_j:V_j \rightarrow V_j'\}_{j \in J}$ of a $\mathcal D^r$ manifold $M$ is equivalent, it is obvious that a subset of the $\mathcal D^r$ manifold $(S,\{\varphi_i\}_{i \in I})$ is definable if and only if it is definable as a subset of the $\mathcal D^r$ manifold $(M, \{\psi_j\}_{j \in J})$.
	
	The Cartesian product of two $\mathcal D^r$ manifold is naturally defined.
	A map $f:S \rightarrow T$ between $\mathcal D^r$ manifolds is \textit{definable} if its graph is definable in $S \times T$.

	(2)
	A definable subset $Z$ of $X$ is called 
	a \textit{$k$-dimensional $\mathcal D^r$ submanifold}
	of $X$ if each point $x \in Z$ there exist an open box $U_x$ of $x$ in $X$
	and a $\mathcal D^r$ diffeomorphism $\phi_x$ from $U_x$ to some open box $V_x$ of
	$F^d$ such that $\phi_x(x)= 0$ and $U_x \cap Y = \phi_x^{-1}(F^k \cap V_x)$. 
	%where $F^k \subset F^d$ is the vectors whose last $(d-k)$ components are zero.
	
	(3) Let $X$ and $Y$ be $\mathcal D^r$ manifolds with $\mathcal D^r$ charts
	$\{\phi_i:U_i \to V_i\}_{i \in A}$ and $\{\psi_j:U'_j \to  V'_j\}_{j \in B}$, respectively. 
	A continuous map $f :X \to Y$ 
	is said to be a \textit{definable $C^r$ map} or a \textit{$\mathcal D^r$ map}
	if for any $i \in A$ and $j \in B$, the image $\phi_i(f^{-1}(V_j') \cap U_i)$ is definable and open in $F^n$ and 
	the map $\psi_j \circ f \circ \phi_i^{-1}:\phi_i(f^{-1}(V_j) \cap U_i) \to F^m$
	is a $\mathcal D^r$ map.
	
	(4) Let $X$ and $Y$ be $\mathcal D^r$ manifolds.
	We say that $X$ is \textit{definably $C^r$ diffeomorphic to} $Y$ or $\mathcal D^r$ diffeomorphic to $Y$
	if there exist $\mathcal D^r$ maps
	$f:X \to Y$ and $h: Y \to X$ such that $f \circ h=\operatorname{id}$ and $h \circ f = \operatorname{id}$.
	
	(5) A $\mathcal D^r$ manifold is called \textit{definably normal} if for any definable closed subset $C$ and definably open subset $U$ of $M$ with $C \subseteq U$, there exists a definable open subset $V$ of $M$ such that $C \subseteq V \subseteq \mycl_M(V) \subseteq U$.
\end{definition}

We begin to prove the imbedding theorem of $\mathcal D^r$ manifold into $F^n$ in the locally o-minimal setting.
It was already demonstrated in the o-minimal setting \cite{Kawa}, but it has a gap in the proof.
A classical proof for that a compact manifold is affine using the partition of unity works in our setting, but we need a slight modification.
We first prepare several technical lemmas.

\begin{lemma}\label{lem:sep0}
	Let $\mathcal F=(F,<,+,\cdot,0,1,\ldots)$ be a definably complete locally o-minimal expansion of an ordered field.
	Let $M \subseteq F^n$ be a $\mathcal D^r$ submanifold.
	Let $X$ and $Y$ be closed definable subsets of $M$ with $X \cap Y = \emptyset$.
	Then, there exists a $\mathcal D^r$ function $f:M \rightarrow [0,1]$ with $f^{-1}(0)=X$ and $f^{-1}(1)=Y$.
\end{lemma}
\begin{proof}
	We may assume that $M$ is closed in $F^n$ by Lemma \ref{lem:closed_mani}. 
	There exist $\mathcal D^r$ functions $g,h:F^n \rightarrow F^n$ with $g^{-1}(0)=X$ and $h^{-1}(0)=Y$ by Theorem \ref{thm:zeroset}. 
	The function $f:F^n \rightarrow [0,1]$ defined by  $f(x)=\frac{g(x)^2}{g(x)^2+h(x)^2}$ satisfies the conditions $f^{-1}(0)=X$ and $f^{-1}(1)=Y$.
	The restriction of $f$ to $M$ satisfies the requirement.
\end{proof}

\begin{corollary}\label{cor:normal}
Consider a definably complete locally o-minimal expansion of an ordered field.
A $\mathcal D^r$ submanifold is definably normal.
\end{corollary}
\begin{proof}
	Let $C$ and $U$ be a definable closed and open subsets of a $\mathcal D^r$ submanifold $M$, respectively, such that $C \subseteq U$.
	There exists a $\mathcal D^r$ function $f:M \rightarrow [0,1]$ with $f^{-1}(0)=C$ and $f^{-1}(1)=M \setminus U$ by Lemma \ref{lem:sep0}.
	Set $V:=f^{-1}([0,\frac{1}{2}))$.
	The definable set is an open subset of $M$.
	We have $C \subseteq V \subseteq \mycl_M(V) \subseteq f^{-1}([0,\frac{1}{2}]) \subseteq U$.
	It means that $M$ is definably normal.
\end{proof}

The following lemma is the partition of unity:
\begin{lemma}[Partition of unity]\label{lem:unity0}
	Let $\mathcal F=(F,<,+,\cdot,0,1,\ldots)$ be a definably complete locally o-minimal expansion of an ordered field.
	Given a $\mathcal D^r$ atlases $\{\varphi_i:U_i \rightarrow U_i'\}_{1 \leq i \leq k}$ of a definably normal $\mathcal D^r$ manifold $M$, 
	there exist nonnegative $\mathcal D^r$ functions $\lambda_i$ on $M$ for all $1 \leq i \leq q$ such that $\sum_{i=1}^q\lambda_i = 1$, the closure of the set $\{x \in M\;|\; \lambda_i(x)>0\}$ is contained in $U_i$ and the family of definable open sets $\{\lambda_i^{-1}((0,\infty))\}_{i=1}^k$ is an open cover of $M$.
\end{lemma}
\begin{proof}
	Set $U_0=\emptyset$.
	By induction on $i$, we construct $D^r$ function $\psi_i:M \to [0,\infty)$ such that 
	\begin{itemize}
		\item $\mysupp (\psi_i) \subseteq U_i$ and
		\item $M=\bigcup_{j=1}^i \psi_j^{-1}((0,\infty)) \cup \bigcup_{j=i+1}^k U_j$
	\end{itemize}
	for each $0 \leq i \leq k$.
	Here, $\mysupp (\psi_i)$ denotes the support of $\psi_i$, which is the closure of the set $\{x \in M\;|\; \psi_i(x)>0\}$.
	
	The case in which $i=0$ is trivial. Set $\psi_0=0$.
	For $i >0$, set $$V_i=\bigcup_{j=1}^{i-1} \psi_j^{-1}((0,\infty)) \cup \bigcup_{j=i+1}^k U_j.$$
	We have $M=U_i \cup V_i$ by the induction hypothesis.
	Put $W_i=M \setminus V_i$ and $W'_i=\varphi_i(M \setminus V_i$).
	Since $M$ is definably normal, there exists a definable open subset $B_i$ of $M$ such that $W_i \subseteq B_i \subseteq \mycl_M(B_i) \subseteq U_i$.
	Set $C_i=U_i \setminus B_i$, $B'_i=\varphi_i(B_i)$ and $C'_i=\varphi_i(C_i)$.
	Both $W'_i$ and $C'_i$ are definable closed subsets of $U'_i$, and they have an empty intersection.
	Therefore, there exists a $\mathcal D^r$ function $f_i$ on $U'_i$ which equals one in $W'_i$ and vanishes in $C'_i$ by Lemma \ref{lem:sep0}.
	%Furthermore, $\partial U'_i \subseteq \partial W_i \cup \partial Z_i$.
	Consider the definable map $\psi_i:M \to F$ given by $$\psi_i(x)=\left\{\begin{array}{ll}f_i(\varphi_i(x)) & \text{if } x \in U_i\\ 0 & \text{otherwise.}\end{array}\right.$$
	We demonstrate that this function satisfies the requirements.
	The equality $M=\bigcup_{j=1}^i \psi_j^{-1}((0,\infty)) \cup \bigcup_{j=i+1}^k U_j$ holds true because $\psi_i$ is positive in $M \setminus V_i$.
	Since $\{x \in U'_i\;|\; \varphi_i(x)>0\} \subseteq B'_i$, we have $\mysupp(\psi_i)
 \subseteq \mycl_M(B_i) \subseteq U_i$.	
	
	The function $\psi_i$ is obviously of class $\mathcal C^r$ by the inclusion $\mysupp(\psi_i) \subseteq U_i$.
	Set $$\lambda_i(x)=\dfrac{\psi_i(x)}{\sum_{i=1}^k \psi_i(x)}$$ for all $1 \leq i \leq k$.
	The functions $\lambda_i$ satisfies the requirements.	
\end{proof}
%
%Recall that a $D^r$ manifold $X$ is \textit{Hausdorff} if
%for any distinct points $x, y \in X$, there exist open sets $U, V$ such that $x \in U$, $y \in V$ and $U \cap V=\emptyset$.  
%We say that $X$ is \textit{regular}
%if for any $x \in X$ and any open subset $U$ of $X$ with $x \in U$,
%there exists an open  subset $V$ of $X$ with $x \in V$ and the closure of $V$ in $X$ is contained in $U$. 

The following is the $\mathcal D^r$ imbedding theorem of $\mathcal D^r$ manifolds:
\begin{theorem}\label{thm:imbedding}
	Let $\mathcal F=(F,<,+,\cdot,0,1,\ldots)$ be a definably complete locally o-minimal expansion of an ordered field.
	Every definably normal $\mathcal D^r$ manifold is definably imbeddable into some $F^n$, and its image is a $\mathcal D^r$  submanifold of $F^n$.
\end{theorem}
\begin{proof}
	Let $M$ be a definably normal $D^r$ manifold with atlas $\{\varphi_i:U_i \to U'_i \subseteq F^{n_i}\}_{i=1}^k$.
	Take functions $\lambda_i:M \to F$ satisfying the conditions in Lemma \ref{lem:unity0}.
	Define the map $\Phi:M \to F^{\sum_{i=1}^kn_i+k}$ by
	$$\Phi(x)=(\lambda_1(x)\varphi_1(x),\ldots, \lambda_k(x)\varphi_k(x),\lambda_1(x), \ldots, \lambda_k(x)).$$
	The map $\Phi$ is well-defined because the support of $\lambda_i$ is contained in $U_i$.
	It is clearly a $\mathcal D^r$ map.
	
	Set $X_i=\lambda_i^{-1}((0,\infty))$ for $1 \leq i \leq k$.
	Note that $X_i$ is a subset of $U_i$.
	We first demonstrate the following claim:
	\medskip
	
	{\textbf{Claim.}} For $1 \leq i \leq k$, each $\Phi(X_i)$ is a $\mathcal D^r$ submanifold and open in $\Phi(M)$. 
	The restriction of $\Phi$ to $X_i$ is a $\mathcal D^r$ diffeomorphism onto its image.
	\begin{proof}[Proof of Claim.]
	Let $\pi_i:\Phi(M) \to \{(\lambda_i(x)\varphi_i(x), \lambda_i(x))\;|\; x \in M\}$ be the canonical projection.
	Let $\tau_i:\pi_i(\Phi(M))  \setminus \{0\}\to \varphi_i(X_i)$ be the $\mathcal D^r$ diffeomorphism given by 
	$\tau_i(y,z)=y/z$.
	Since $\varphi_i(X_i)$ is a definable open subset of a $\mathcal D^r$ submanifold $U'_i$, it is also a $D^r$ submanifold.
	We have $\Phi(X_i)=\Phi(M) \cap \pi_i^{-1}(F^{n_i} \times (F \setminus \{0\}))$.
	Therefore, $\Phi(X_i)$ is open in $\Phi(M)$.
	It is also trivial that $\Phi(X_i)$ is the graph of a $D^r$ map defined on $\pi_i(\Phi(X_i))=\pi_i(\Phi(M))  \setminus\{0\}$ after an appropriate permutation of coordinates.
	Therefore, it is a $\mathcal D^r$ submanifold.
	Since $\tau_i \circ \pi_i \circ \Phi|_{X_i}=\varphi_i|_{X_i}$ and the restriction of $\pi_i$ to $\Phi(X_i)$ 
	is a $\mathcal D^r$ diffeomorphism, 
	the restriction of $\Phi$ to $X_i$ is also a $\mathcal D^r$ diffeomorphism onto its image.
    \end{proof}
	
	We can get several corollaries from the claim.
	Firstly, the claim implies that $\Phi$ is a local $\mathcal D^r$ diffeomorphism.
	Secondly, by the definition of $\mathcal D^r$ submanifolds, 
	$\Phi(M)$ is a $\mathcal D^r$ submanifold because there is a definable open cover $\{\Phi(X_i)\}_{i=1}^k$ of $\Phi(M)$ each element of which is a $\mathcal D^r$ submanifold.
	
	We next demonstrate that $\Phi$ is an imbedding.
	We have only to show that $\Phi$ is a homeomorphism onto its image.
	The first task is to prove that $\Phi$ is injective.
	Let $x,y \in M$ with $\Phi(x)=\Phi(y)$.
	There is an $1 \leq i \leq k$ such that $\lambda_i(x)>0$.
	We also have $\lambda_i(y)>0$.
	They imply that $x \in U_i$ and $y \in U_i$.
	We immediately get $\varphi_i(x)=\varphi_i(y)$.
	It means that $x=y$.
	
	We next prove that the inverse of $\Phi$ is continuous.
	For that, we have only to prove that $\Phi$ is an open map.
	Let $U$ be a definable open subset of $M$.
	We demonstrate that $\Phi(U)$ is open in $\Phi(M)$.
	Take an arbitrary point $y \in \Phi(U)$.
	Take the unique point $x \in M$ such that $y=\Phi(x)$.
	Take a sufficiently small definable open neighborhood $V$ of $x$ in $M$ such that $V \subseteq U$.
	It is possible because $M$ is definably normal.
	We may assume that the restriction $\Phi|_{V}$ is a diffeomorphism by the claim.
	The definable set $\Phi(V)$ is an open neighborhood of $x$ because $\Phi|_{V}$ is a diffeomorphism.
	On the other hand, injectivity of $\Phi$ implies that $\Phi(V) \subseteq \Phi(U)$.
	We have demonstrated that $\Phi(U)$ is open in $\Phi(M)$.
\end{proof}

\begin{remark}
	We assume that $\mathcal M$ is a definably complete locally o-minimal expansion of an ordered field in this subsection.
	Every assertions in this subsection hold when $\mathcal M$ is a definably complete expansion of an ordered field in which an arbitrary definable closed set is the zero set of a $\mathcal D^r$ function. 
\end{remark}

\subsection{Definable quotient of groups}\label{sec:group_quotient}

As an application of Theorem \ref{thm:imbedding}, we want to prove that the quotient group of a definable $\mathcal C^r$ group by a definable normal subgroup. 
We first define definable $\mathcal C^r$ groups.

\begin{definition}
	Consider an expansion of a dense linear order without endpoints.
	A \textit{definable group} is a group $(G,\cdot,e)$ such that $G$  is definable and both the multiplication $(a,b) \mapsto a \cdot b$ and the inverse $a \mapsto a^{-1}$ are definable maps.
	We define a \textit{definable subgroup} of a definable group naturally.
	A definable group is called a \textit{definable $\mathcal C^r$ group} or a \textit{$\mathcal D^r$ group} if both the multiplication and the inverse are of class $\mathcal C^r$.
	
	A \textit{definable equivalence relation} $E$ on a definable set $X$ is a definable subset of $X \times X$ such that the relation $\sim$ defined by $a \sim b \Leftrightarrow (a,b) \in E$ is an equivalence relation.
	Let $G$ be a definable group and $H$ be its definable subgroup.
	The relation $E_H$ given by $E_H=\{(g,hg) \in G \times G\;|\; g \in G, h \in H\}$ is a definable equivalence relation.
\end{definition}

We first demonstrate that a $\mathcal D^r$ group is a $\mathcal D^r$ submanifold after some preparations.

\begin{definition}
	Consider a definably complete locally o-minimal expansion of an ordered field $\mathcal F=(F,<,+,\cdot,0,1,\ldots)$.
	Let $X$ and $Y$ be definable subsets of a common ambient space.
	The definable set $X$ is \textit{large} in $Y$ if $\dim(Y \setminus X)<\dim Y$.
\end{definition}

We next recall Wencel's result \cite[Corollary 2.5]{Wencel}.

\begin{proposition}\label{prop:finite_cover}
	Consider a definably complete locally o-minimal expansion of an ordered field $\mathcal F=(F,<,+,\cdot,0,1,\ldots)$.
	Let $G$ be a definable group of dimension $d$ and $U$ be a definable subset of $G$ which is large in $G$.
	There exists finitely many $g_1, \ldots, g_{d+1} \in G$ such that $G=\bigcup_{i=1}^{d+1} g_iU$.
\end{proposition}
\begin{proof}
	It follows from \cite[Corollary 2.5]{Wencel}.
	In \cite[Corollary 2.5]{Wencel}, this proposition is given in more general context under the assumption that the dimension function satisfies several conditions given in \cite{Wencel}.
	Our dimension function satisfies these conditions by \cite[Proposition 2.8]{FKK}.
\end{proof}

\begin{lemma}\label{lem:large_decomposition}
	Let $\mathcal{F}=(F, <, +, \cdot, 0, 1, \dots)$ be a definably complete locally o-minimal expansion of an ordered field.
	Let $X$ be a definable subset of $F^n$.
	Then there exist  quasi-special $\mathcal C^r$ submanifolds $U_1, \dots, U_m$ contained in $X$ such that the union
	$U_1 \cup \dots \cup U_m$ is a $\mathcal D^r$ submanifold of $X$ and 
	$U_1 \cup \dots \cup U_m$ is large in $X$.
\end{lemma}
\begin{proof}
	Put $d=\dim X$ and $\Pi_{n, d}:=\{\pi:F^n \to F^d\;|\; \pi \mbox{ is a coordinate projection}\}$.
	It is obvious that $\Pi_{n, d}$ is a finite set.
	For any $\pi \in \prod_{n, d}$, by \cite[Theorem 2.5]{FKK} and \cite[Lemma 4.2]{Fuji4}, $$U_{\pi}=\{x \in X\;|\;x \text{ is } (X, \pi)\text{-normal}\}$$ is
	a quasi-special submanifold of dimension $d$ if $U_{\pi} \neq \emptyset$.
	For every $x \in U_{\pi}$,
	there exists an open box $B \subset F^n$ with $x \in B$
	such that, after changing coordinates if necessary, 
	$B \cap X$ is the graph of a definable continuous map defined on $\pi (B)$.
	
	We fix $\pi \in \Pi_{n,d}$ such that $U_{\pi}$ is not empty.
	By permuting the coordinates if necessary, we may assume that $\pi$ is the projection onto the first $d$ coordinates.
	Let $Z_\pi$ be the set of points $x \in U_{\pi}$ such that, for any open box $B$ with $x \in B$, the intersection $B \cap X$ is not the graph of a $\mathcal D^r$ map defined on $\pi(B)$.
	We want to demonstrate that $\dim Z_{\pi}<\dim U_{\pi}$.
	Assume for contradiction that $\dim Z_{\pi}=\dim U_{\pi}$.
	By Proposition \ref{prop:pre}(7),
	there exists a point $x \in Z_{\pi}$ such that for any open box $U \subset M^n$ with $x \in U$,
	$\dim (Z_{\pi} \cap U)=\dim U_{\pi}$.
	By Proposition \ref{prop:pre}(6) and $Z_{\pi} \subseteq U_{\pi}$,
	the projection image $\pi (Z_{\pi} \cap U)$ has a nonempty interior.
	We can take an open box $V$ contained in $\pi (Z_{\pi} \cap U)$. 
	Set $U'=\pi^{-1}(V) \cap U$.
	The intersection $U' \cap Z_{\pi}=U' \cap U_{\pi}$ is the graph of a definable map on $V$ such that the map is not of class $\mathcal C^r$ at any point in $V$.
	It contradicts Proposition \ref{prop:cr_pre}.
	We have demonstrated that $\dim Z_{\pi}<\dim U_{\pi}$.
	
	Set $T_{\pi}=U_{\pi} \setminus \mycl(Z_{\pi})$.
	We have $\dim (U_{\pi}-T_{\pi})<d$ by Proposition \ref{prop:pre}(3).
	Let $U_1, \dots, U_m$ be the enumeration of the family $\{U_{\pi}-T_{\pi}\;|\;\pi \in \prod_{n, d} \text{ and } U_{\pi} \neq \emptyset\}$.
	We can prove $\dim (X \setminus \cup_{\pi \in \prod_{n,d}}U_{\pi})<d$ similarly to the proof of \cite[Proposition 4.3]{Fuji4}.
	We omit the details.
	Since $\dim (U_{\pi}-T_{\pi})<d$, we have $\dim (X-\cup_{i=1}^m U_i)<d$.
	Each $U_i$ is a quasi-special $\mathcal C^r$ submanifold by the definition of $U_i$.
	At any point $x \in U_1 \cup \dots \cup U_m$, we can take an open box $B$ such that $X \cap B = X \cap U_i$ for some $1 \leq i \leq m$ and the intersection $X \cap U_i$ is the graph of a $\mathcal D^r$ map after changing the coordinates appropriately.
	It means that $U_1 \cup \dots \cup U_m$ is also a $\mathcal D^r$ submanifold.
\end{proof}

\begin{corollary}\label{cor:group_is_manifold}
	Consider a definably complete locally o-minimal expansion of an ordered field.
	A $\mathcal D^r$ group is a $\mathcal D^r$ manifold.
\end{corollary}
\begin{proof}
	Let $G$ be a $\mathcal D^r$ group and set $d=\dim G$.
	There are finitely many quasi-special $\mathcal C^r$ submanifolds $U_1, \ldots, U_m$ and $g_1,\ldots, g_{d+1} \in G$ such that $$G=\bigcup_{i=1}^{d+1} \bigcup_{j=1}^m g_iU_j$$ by Lemma \ref{lem:large_decomposition} and Proposition \ref{prop:finite_cover}.
	Let $\varphi_{ij}:g_iU_j \to U_j$ be the $\mathcal D^r$ map given by $g_ix \mapsto x$.
	The family $\{\varphi_{ij}\}_{1 \leq i \leq d+1, 1 \leq j \leq m}$ gives a $\mathcal D^r$ atlas of $G$.
\end{proof}

Our next task is to prove that a definable subgroup of a $\mathcal D^r$ group $G$ is a closed $\mathcal D^r$ subgroup.
We first prove a more general fact. 

\begin{theorem}\label{thm:open_cover}
	Consider a definably complete locally o-minimal expansion of an ordered field.
	Let $r>0$.
	A $\mathcal D^r$ submanifold has a definable open cover whose members are quasi-special $\mathcal C^r$ submanifolds.
\end{theorem}
\begin{proof}
	The implicit function theorem for $\mathcal C^1$ maps definable in o-minimal structures is found in \cite[p.112]{vdD}.
	Only intermediate value theorem is used for the proof and it holds true for definably complete structures.
	Therefore, the implicit function theorem for $\mathcal C^1$ maps is available in our setting.
	In addition, when the definable map in consideration is of class $\mathcal C^r$, the inverse map is also of class $\mathcal C^r$ because the Jacobian of the inverse map is the inverse of the Jacobian of the original map.
	We use this fact.
	
	Let $\mathcal{F}=(F, <, +, \cdot, 0, 1, \dots)$ be a definably complete locally o-minimal expansion of an ordered field.
	We fix a $\mathcal D^r$ submanifold $X$ of $F^n$ of dimension $d$.
	Let $\Pi_{n,d}$ be the set of coordinate projections from $F^n$ onto $F^d$.
	We first show the following claim:
	\medskip

	\textbf{Claim 1.} 
	Let $x \in X$ and $\pi \in \Pi_{n,d}$.
	If $\pi(T_xX)=F^d$, the point $x$ is $(X,\pi)$-$\mathcal C^r$-normal.
	\begin{proof}[Proof of Claim 1]
	Let $\varphi:U \to V$ be an atlas of $X$ at $x$.
	In other word, $U$ and $V$ are definable open subsets of $F^n$, $\varphi$ is a  $\mathcal D^r$ diffeomorphism, $\varphi(0)=x$ and $\varphi(U \cap H)=X \cap V$, where $H=\{x =(x_1,\ldots, x_n) \in F^n\;|\;x_{d+1}=\dots =x_n=0\}$. 
	By the definition of tangent spaces and the assumption that $\pi(T_xX)=F^d$, the matrix $d_0(\pi \circ \varphi|_{H})$ is invertible.
	By the inverse function theorem described above, there exist a definable open neighborhood $W$ of $\pi(x)$ and $\mathcal D^r$ map $g:W \to F^d$ such that $\pi \circ \varphi \circ g= \operatorname{id}$ on $W$.
	Set $\psi=\varphi \circ g: W \to V$.
	We have $\psi(W) \subseteq X \cap V$ and $\pi \circ \psi = \operatorname{id}$.
	It means that the point $x$ is $(X,\pi)$-$\mathcal C^r$-normal.
	We have proven the claim.
	\end{proof}
	
	We next demonstrate the following claim:
	\medskip
	
	\textbf{Claim 2.} Let $\pi \in \Pi_{n,d}$.
	Set $T_{\pi}:=\{x \in X\;|\; \pi(T_xX)=F^d\}$.
	The definable set $T_{\pi}$ is open and a $\pi$-quasi-special $\mathcal C^r$ submanifold.
	
	\begin{proof}[Proof of Claim 2]
	It is obvious that $T_\pi$ is definable.
	We consider the map $\Delta:X \to \mathbb G_{n,d}$ given by $\Delta(x)=T_xX$.
	Here, $\mathbb G_{n,d}$ denotes the Grassmannian, which is algebraic.
	The map $\Delta$ is definable and continuous.
	The subset $U_\pi=\{L \in \mathbb G_{n,d}\;|\; \pi(L)=F^d\}$ is an semialgebraic open subset of $\mathbb G_{n,d}$.
	%Note that any semialgebraic set is always definable.
	Therefore, $T_{\pi}=\Delta^{-1}(U_{\pi})$ is also open.
	
	By Claim 1, any point in $T_{\pi}$ is $(X,\pi)$-$\mathcal C^r$-normal.
	Since $T_{\pi}$ is definable and open, any point in $T_{\pi}$ is $(T_{\pi},\pi)$-$\mathcal C^r$-normal.
	We can demonstrate that $T_{\pi}$ is $\pi$-quasi-special $\mathcal C^r$-submanifold in the same manner as \cite[Lemma 4.2]{Fuji4}.
	We have demonstrated Claim 2.
	\end{proof}
	
	We are now ready to complete the proof of the theorem. 
	Since $\Pi_{n,d}$ is a finite set, thanks to Claim 2, we have only to demonstrate $X \subseteq \bigcup_{\pi \in \Pi_{n,d}} T_{\pi}$.
	Let $x$ be an arbitrary point of $X$.
	Let $e_1,\ldots, e_d \in F^n$ be the basis of the vector space $T_xX$.
	Set $A=(e_1,\ldots,e_d)$, which is $(n,d)$-matrix.
	We can take $1 \leq i_1 < \dots <i_d \leq n$ such that the square matrix $B$ constructed from the $i_1, \ldots i_d$-th rows of $A$ is invertible.
	Let $\pi:M^n \to M^d$ be the coordinate projection given by $\pi(x_1,\ldots, x_n)=(x_{i_1},\ldots, x_{i_d})$.
	We obviously have $\pi(T_xX)=F^d$.
	It means that $x \in T_{\pi}$.
	We have shown the theorem.
\end{proof}

\begin{proposition}\label{prop:subgroup}
	Consider a definably complete locally o-minimal expansion of an ordered field $\mathcal{F}=(F, <, +, \cdot, 0, 1, \dots)$.
	Let $G$ be a $\mathcal D^r$ group and $H$ be a definable subgroup of $G$.
	Then $H$ is a closed in $G$ and $\mathcal D^r$ submanifold of $G$.
\end{proposition}
\begin{proof}
	A definably complete locally o-minimal structure is a first-order topological structure in the sense of \cite{P87}.
	Any definable set is always constructible by \cite[Theorem 2.5]{FKK} and \cite[Corollary 3.10]{Fuji4}.
	The definable subgroup $H$ is a closed subgroup by \cite[Proposition 2.7]{P87} together with the above facts.
	\medskip
	
	{\textbf{Claim.}} Let $X$ be a $\mathcal D^r$ submanifold of $F^n$ of dimension $m$ and 
	$Y \subseteq X$ be a definable closed subset of $X$ with $\dim Y=k$.
	Then there exists an open box $U$ such that $U \cap Y$ is a $\mathcal D^r$ submanifold of $U \cap X$.
	
	\begin{proof}[Proof of Claim]
	By Theorem \ref{thm:open_cover}, there exist finitely many quasi-special $C^r$ submanifolds $U_1, \dots, U_s$ such that $\{U_1, \dots, U_s\}$ is an open cover of $X$.
	Take $i$ such that $\dim (Y \cap U_i)=k$.
	Such an $i$ exists by Proposition \ref{prop:pre}(2).
	Considering $U_i$ instead of $X$, and permuting coordinates if necessary,
	we may assume that 
	$X$ is a $\pi$-quasi-special $\mathcal C^r$ submanifold and $\pi$ is the projection onto first $m$ coordinates.
	By Proposition \ref{prop:pre}(7),
	there exists an $x' \in Y$ such that for any open box with $x' \in B$,
	$\dim (Y \cap B)=k$.
	Take a sufficiently small $B$.
	By the definition of quasi-special $\mathcal C^r$ manifolds,
	we may assume that 
	$X \cap B$ is the graph of a $\mathcal D^r$ map defined on $C=\pi (B)$.
	Replacing $X$ by $X \cap B$,
	we may assume that $X$ is the graph of a $\mathcal D^r$ map $\psi$ defined on $C$.
	By Proposition \ref{prop:pre}(5)(6), we have $\dim \pi (Y)=k$.
	By the definition of dimension,
	there exists a coordinate projection $\pi':F^m \to F^k$ such that 
	$\myint(\pi'(\pi(Y))) \neq \emptyset$.
	Permuting the coordinates once again, we may assume that $\pi'$ is the projection of $F^m$ onto the first $k$ coordinates.
	
	Put $S=\{y \in \pi'(\pi(Y))|\dim((\pi')^{-1}(y) \cap \pi (Y))=0\}$.
	By Proposition \ref{prop:pre}(6),
	we have $\dim S=k$.
	In particular, the interior of $S$ is not an empty set.
	Set $$T=\{y \in S\;|\; \text{For any }x \in \pi (Y) \text{ with } \pi'(x)=y, x \text{ is } (\pi(Y), \pi')\text{-}\mathcal C^r \text{normal}\}.$$
	Similarly to \cite[Lemma 4.3]{Fuji4}, we get $\myint (T) \neq \emptyset$.
	By Proposition \ref{prop:definable_choice}, there exists a definable map $\tau:T \to (\pi')^{-1}(T) \cap \pi (Y)$.
	By Proposition \ref{prop:cr_pre}, 
	there exists an open box $V$ such that $\tau|_V$ is of class $\mathcal C^r$.
	Similarly to \cite[Lemma 4.2]{Fuji4},
	$\pi (Y) \cap (\pi')^{-1}(V)$ is a $\pi'$-quasi-special $C^r$ submanifold.
	
	There exists an open box $W \subseteq F^m$ such that $W \cap \pi (Y)$ is the graph of $\tau|_{\pi'(W)}$.
	Put $U:=\pi^{-1}(W)$.
	The intersection $U \cap Y$ is the graph of $(\tau|_{\pi'(W)}, \psi|_{W \cap \pi (Y)})$.
	The other intersection $U \cap X$ is the graph of $\psi|_{W \cap \pi (Y)}$.
	Thus $U \cap Y$ is a $D^r$ submanifold of $U \cap X$.
	We have proven the claim.
	\end{proof}
	
	By Claim,
	there exist $g \in H$ and open box $U$ with $g \in U$
	such that 
	$H \cap U$ is a $\mathcal D^r$ submanifold of $G \cap U$.
	For any $h \in H$, the map $\phi_h:G \to G$ given by $\phi_h(t)=tg^{-1}h$ is a  $\mathcal D^r$ diffeomorphism and
	$\phi_h(H)=H$.
	Then $h \in \phi_h (U)$ and $H \cap \phi_h (U)$ is a $D^r$ submanifold of $g \cap \phi_h(U)$.
	By the definition of $D^r$ submanifolds,
	$H$ is a $D^r$ submanifold of $G$.
\end{proof}

We prove the existence and uniqueness of $\mathcal D^r$ group structure of a definable group.
\begin{definition}
	Consider an expansion of an ordered field.
	Let $G$ be a definable group.
	A \textit{definable $\mathcal C^r$ group structure} or \textit{$\mathcal D^r$ group structure} on $G$ is a pair of a $\mathcal D^r$ group $H$ and a definable group isomorphism $\iota: G \to H$. 
	Two $\mathcal D^r$ group structures $(H_1,\iota_1)$ and $(H_2,\iota_2)$ are \textit{equivalent} if the composition $\iota_2 \circ \iota_1^{-1}$ is a $\mathcal D^r$ diffeomorphism.
\end{definition}

\begin{theorem}\label{thm:dr_structure}
	Consider a definably complete locally o-minimal expansion of an ordered field $\mathcal F=(F,<,+,\cdot,0,1,\ldots)$.
	Let $G$ be a definable group.
	There exists a $\mathcal D^r$ group structure on $G$ and it is unique up to equivalence.
\end{theorem}
\begin{proof}
	We first prove the existence.
	Using Proposition \ref{prop:cr_pre} instead of continuity property in \cite{Wencel}, we can prove that there exists a definable subset $V$ of $G$ satisfying the following conditions in the same manner as \cite[Theorem 3.5]{Wencel}.
	\begin{itemize}
		\item $V$ is large and open in $G$;
		\item The inversion is a $\mathcal D^r$ map from $V$ onto $V$;
		\item For any $a, b \in G$, the set $$Z(a,b):=\{x \in V\;|\; axb \in V\}$$ is open in $V$ and the map $x \mapsto axb$ is $\mathcal D^r$;
		\item For any $a,b \in G$, the set $$Z'(a,b)=\{(x,y) \in V \times V\;|\; axby \in V\}$$ is open in $V \times V$ and the map $F_{a,b}:Z'(a,b) \ni (x,y) \mapsto axby \in V$ is $\mathcal D^r$. 
	\end{itemize}
Apply Proposition \ref{prop:multi1} to $V$ and decompose $V$ into finitely many special $\mathcal C^r$ manifolds.
Let $\{U_1,\ldots, U_m\}$ be the family of special $\mathcal C^r$ manifolds of dimension $=\dim G$ in the decomposition.
The definable set $U_i$ and open in $V$ for each $1 \leq i \leq m$.
The union $\bigcup_{i=1}^m U_i$ is large in $V$ and it is also large in $G$.
Apply Proposition \ref{prop:finite_cover} to the union $\bigcup_{i=1}^m U_i$.
There are finitely many $a_1, \ldots, a_n \in G$ with $G=\bigcup_{i=1}^m\bigcup_{j=1}^n a_jU_i$.

We define the $\mathcal D^r$ manifold $M$ as follows:
The map $\varphi_{ij}:a_jU_i \to U_i$ is defined by $a_jU_i \ni a_jx \mapsto x \in U_i$. 
Note that $U_i$ is a $\mathcal D^r$ submanifold by the definition of special $\mathcal C^r$ submanifolds.
The pair $M=(G,\{\varphi_{ij}\}_{1 \leq i\leq m,1 \leq j \leq n})$ is a $\mathcal D^r$ manifold.

We want to show that the multiplication and the inversion in $G$ are $\mathcal D^r$ maps defined on $M \times M$ and $M$, respectively.
We only consider the multiplication.
The proof is similar for the inversion.
We denote the multiplication by $f:G \times G \to G$.
Fix $g_1, g_2 \in G$.
There exist $1 \leq i \leq m$ and $1 \leq j \leq n$ with $g_1g_2 \in a_jU_i$.
We also choose $1 \leq i_k \leq m$ and $1 \leq j_k \leq n$ so that $g_k=a_{j_k}U_{i_k}$ for $k=1,2$.
Take a sufficiently small definable open neighborhood $W_k$ of $g_k$ contained in $a_{j_k}U_{i_k}$ for $k=1,2$.
We have only to show that $g:\varphi_{i_1j_1}(W_1) \times \varphi_{i_2j_2}(W_2) \to U_i$ given by $g(x_1,y_2)=\varphi_{ij}(f(\varphi_{i_1j_1}^{-1}(x_1),\varphi_{i_2j_2}^{-1}(x_2)))$ is $\mathcal D^r$.

Set $u_k=a_{j_k}^{-1}g_k \in V$ for $k=1,2$.
The set $Z''(a_j^{-1}g_1,g_2):=Z'(a_j^{-1}g_1,g_2) \cap F_{a_j^{-1}g_1,g_2}^{-1}(U_i)$ is open and contains $(e,e)$.
The restriction of $F_{a_j^{-1}g_1,g_2}$ to $Z''(a_j^{-1}g_1,g_2)$ is of class $\mathcal C^r$.
The map $x \mapsto u_k^{-1}x$ is also $\mathcal D^r$ near the point $u_k$ because $u_k \in Z(u_k^{-1},e)$ for $k=1,2$.
We have $g(x_1,x_2)=F_{a_j^{-1}g_1,g_2}(u_1^{-1} x_1,u_2^{-1} x_2)$.
It implies that $g$ is $\mathcal D^r$ near $(u_1,u_2)$.

We have demonstrated that $M$ is $\mathcal D^r$ and multiplication and inversion in $G$ induce $\mathcal D^r$ maps on $M$.
The $\mathcal D^r$ manifold $M$ is $\mathcal D^r$ isomorphic to a $\mathcal D^r$ submanifold $H$ by Theorem \ref{thm:imbedding}.
Let $\iota:M \to H$ be the imbedding.
We can naturally define the group operations on $H$ and they are $\mathcal D^r$.
Therefore, the pair $(H,\iota)$ is a $\mathcal D^r$ group structure on $G$.

We next prove the uniqueness up to equivalence.
Let $(H_1,\iota_1)$ and $(H_2,\iota_2)$ be $\mathcal D^r$ group structures on $G$.
Consider the map $f:=\iota_2 \circ \iota_1^{-1}:H_1 \to H_2$.
There exists a definable subset $U$ of $H_1$ such that the restriction of $f$ to $U$ is $\mathcal D^r$ and $U$ is large in $H_1$ by Proposition \ref{prop:cr_pre}.
Take an arbitrary element $h \in H_1$ and $g \in U$.
We have $f(x)=f(h) \cdot f(g^{-1}) \cdot f|_{U}(gh^{-1}x) $ for each $x \in H_1$ sufficiently close to $h$ because $f$ is a group isomorphism.
It implies that $f$ is $\mathcal C^r$ near the point $h$ because $H_1$ and $H_2$ are $\mathcal D^r$ groups.
Since $h$ is arbitrary, the map $f=\iota_2 \circ \iota_1^{-1}$ is $\mathcal D^r$.
In the same way, the composition $\iota_1 \circ \iota_2^{-1}$ is also $\mathcal D^r$.
It means that $(H_1,\iota_1)$ and $(H_2,\iota_2)$ are equivalent.
\end{proof}

We begin with the proof for the existence of the definable $\mathcal C^r$ quotient of a $\mathcal D^r$ group by a definable subgroup.

We recall the notion of definably identifying maps.
\begin{definition}
	Consider an expansion of a dense linear order without endpoints $\mathcal M=(M,<,\ldots)$.
	Let $X \subseteq M^m$ and $Y \subseteq M^n$ be definable sets and $f:X \to Y$ be a definable continuous map.
	%	The map $f$ is \textit{definably proper} if, for any closed bounded definable set $K$ in $M^n$ with $K \subseteq Y$, the inverse image $f^{-1}(K)$ is bounded and closed in $M^m$.
	The map $f$ is \textit{definably identifying} if it is surjective and, for each definable subset $K$ in $Y$, $K$ is closed in $Y$ whenever $f^{-1} (K)$ is closed in $X$.
%	
%	A \textit{definable curve} in $X$ is a definable continuous map on an open interval $(c,d)$ into $X$.
%	We also call its image a \textit{definable curve}.
%	This abuse of terminology will not confuse readers.
%	In some cases, continuity is not required in the definition of definable curves, but we require it in this paper. 
%	When we consider a definably complete locally o-minimal expansion of an ordered group, the curve $\gamma:(0,\varepsilon) \to X$ has at most one point $x$ in $M^m$ such that, for any $0<\delta<\varepsilon$,  any open neighborhood of $x$ in $X$ intersect with the definable curve $\gamma((0,\delta))$ by Proposition \ref{prop:limit}. 
%	The notation $\lim_{t \to 0}\gamma(t)$ denotes this point if it exists.   
%	The definable curve is \textit{completable in $X$} if $\lim_{t \to 0}\gamma(t)$ exists and belongs to $X$.
\end{definition}

\begin{definition}
	Consider a structure $\mathcal M=(M,\ldots)$, a definable set $X$ and a definable equivalence relation $E$ on $X$.
	A \textit{definable quotient of $X$ by $E$} is a definably identifying definable surjective continuous map $f:X \to Y$ such that $f(x)=f(x')$ if and only if $(x,x') \in E$.
	In addition, when both $X$ and $Y$ are $\mathcal D^r$ submanifolds and $f$ is a $\mathcal D^r$ map, we call it a \textit{definable $\mathcal C^r$ quotient of $X$ by $E$} or a \textit{$\mathcal D^r$ quotient of $X$ by $E$} 
	
	We consider the case in which a definable group $G$ acts on a definable set $X$.
	Assume that the action $G \times X \to X$ is a $\mathcal D^r$ map.
	A $\mathcal D^r$ quotient of $X$ by $G$ is defined as the $\mathcal D^r$ quotient of $X$ by the definable equivalence relation $E_{G,X}:=\{(x,gx) \in X \times X\;|\; x \in X, g \in G\}$.
\end{definition}

Our final result is as follows:

\begin{theorem}\label{thm:quotient_group}
	Consider a definably complete locally o-minimal expansion of an ordered field.
	Let $0 \leq r < \infty$.
	Let $G$ be a $\mathcal D^r$ group and $H$ be a definable subgroup of $G$. 
	Assume that $G$ is bounded and closed in the ambient space.
	Then there exist a $\mathcal D^r$ submanifold $X$ of dimension $\dim G-\dim H$ and a $\mathcal D^r$ quotient $\iota:G \to X$ of $G$ by $H$.
	
	In addition, if $H$ is a normal subgroup of $G$, there exist $\mathcal D^r$ maps $\mymult:X \times X \to X$ and $\myinv:X \to X$ such that $\mymult_{G/H}(\iota(g_1),\iota(g_2)) =\iota(g_1g_2)$ and $\myinv_{G/H}(\iota(g))=\iota(g^{-1})$ for $g,g_1,g_2 \in G$.
	In other word, the definable set $X$ is a $\mathcal D^r$ group and it is isomorphic to the quotient group $G/H$ as a group.  
\end{theorem}
\begin{proof}
	Let $\mathcal F=(F,<,+,\cdot,0,1,\ldots)$ be a definably complete locally o-minimal expansion of an ordered field.
	We assume that $G$ is a definable subset of $F^n$.
	Set $d_G=\dim G$, $d_H=\dim H$ and $d=d_G-d_H$.
	Consider the definable equivalence relation $$E=\{(g,gh) \in G \times G\;|\; g \in G, h \in H\}.$$
	Let $p_i:G \times G \to G$ be the canonical projections onto the $i$-th factor.
	Set $q_i:=p_i|_E$ for each $i=1,2$.
	Note that, for any definable subsets $A$ and $B$ of $G$, the map $q_1|_{E \cap (A \times B)}$ is a definable open map.
	In fact, let $U$ be a definable open subset of $E \cap (A \times B)$.
	Take an arbitrary point $x \in q_1(U)$.
	We can choose $y \in A$ such that $(x,y) \in E \cap (A \times B)$.
	Since $U$ is open, there exists a definable open neighborhood $x$ of $V$ in $A$ such that $V \times \{y\} \subseteq U$.
	We have $V \subseteq q_1(U)$.
	It means that $q_1(U)$ is open and $q_1|_{E \cap (A \times B)}$ is a definable open map.
	
	By Proposition \ref{prop:definable_choice}, there exists a definable map $\tau:G \to G$ such that the graph of $\tau$ is contained in $E$ and, for $g_1,g_2 \in G$, the equality $\tau(g_1)=\tau(g_2)$ holds true if and only if $g_2g_1^{-1} \in H$.
	We want to construct a definable open subset $U$ of $\tau(G)$ such that 
	\begin{itemize}
		\item $U$ is large in $\tau(G)$ and 
		\item $\tau$ is of class $\mathcal C^r$ on $V:=\tau^{-1}(U)$.
	\end{itemize}
	Firstly, we have $\dim \tau^{-1}(y)=d_H$ for any $y \in \tau(G)$ by Proposition \ref{prop:pre}(5).
	We have $\dim \tau(G)=d$ by Proposition \ref{prop:pre}(6).
	Let $U'$ be the set of points at which $\tau$ is of class $\mathcal C^r$.
	It is large in $G$ by Proposition \ref{prop:cr_pre}.
	If $\tau$ is of class $\mathcal C^r$ at $x$, it is also of class $\mathcal C^r$ at $xh$ for each $h \in H$.
	In fact, take a definable open neighborhood $N$ of $x$ in $G$ such that  $\tau$ is of class $\mathcal C^r$ on $N$.
	Let $m_h:G \to G$ be the $\mathcal D^r$ diffeomorphism defined by $m_h(g)=gh$.
	The definable set $m_h(N)$ is a neighborhood of $xh$ and $\tau = \tau \circ m_h$ on $m_h(N)$.
	It implies that $\tau$ is also of class $\mathcal C^r$ at $xh$.
	
	We get $U'=\tau^{-1}(\tau(U'))$.
	This equality as well as Proposition \ref{prop:pre}(6) implies that the image $\tau(U')$ is large in $\tau(G)$.
	The interior $\myint_{\tau(G)}(\tau(U'))$ of $\tau(U')$ in $\tau(G)$ is also large in $\tau(G)$ by Proposition \ref{prop:pre}(3).
	There exists a $\mathcal D^r$ submanifold $U$ of $F^n$ of dimension $d$ which is contained and large in $\myint_{\tau(G)}(\tau(U'))$ by Lemma \ref{lem:large_decomposition}.
	The definable subset $U$ is open in $\tau(G)$ by the definition.
	Since $U$ is large and open in $\tau(G)$, the inverse image $V:=\tau^{-1}(U)$ is large and open in $G$ by Proposition \ref{prop:pre}(6).
	The definable map $\tau$ is of class $\mathcal C^r$ on $V$.
	We have constructed the desired definable sets $U$ and $V$.
	
	We temporally fix  two points $u,v \in G$.
	Set $$U_{u,v}:=q_1(E \cap (uU \times vU)).$$
	It is open in $uU$ because $q_1|_{E \cap(uU \times vU)}$ is an open map.
	We define $U_{v,u}:=q_1(E \cap (vU \times uU))$ similarly.
	It is obvious that $U_{v,u}=q_2(E \cap (uU \times vU))$.
	We next define the definable map $$\varphi_{u,v}:U_{u,v} \to U_{v,u}.$$
	The definable set $vU$ intersects with the orbit $gH$ at only one point $g'$  for any $g \in G$ by the definition of $U$.
	It is obvious $g' \in U_{v,u}$ when $g \in U_{u,v}$.
	We define $\varphi_{u,v}(g)=g'$. 
	We want to show that $\varphi_{u,v}$ is of class $\mathcal C^r$.
	We obviously have $U_{u,v}=vU_{v^{-1}u,e}$, $U_{v,u}=vU_{e,v^{-1}u}$ and $\varphi_{u,v}(g)=v\varphi_{v^{-1}u,e}(v^{-1}g)$ for any $g \in U_{u,v}$.
	Therefore, we have only to demonstrate that $\varphi_{u,v}$ is of class $\mathcal C^r$ when $v=e$.
	On the other hand, the definable map $\varphi_{u,e}$ coincides with the restriction of $\tau$ on $U_{u,e}$, which is of class $\mathcal C^r$.
	It is obvious that $\varphi_{v,u}$ is the inverse of $\varphi_{u,v}$.
	We have demonstrated that $\varphi_{u,v}:U_{u,v} \to U_{v,u}$ are $\mathcal D^r$ diffeomorphism.
	
	There exist $g_1,\ldots, g_{d_G+1} \in G$ such that $G=\bigcup_{i=1}^{d_G+1} g_iV$ by Proposition \ref{prop:finite_cover}.
	Set $U_i=g_iU$ and $V_i=g_iV$ for each $1 \leq i \leq d_G+1$.
	Put $U_{ij}=U_{g_i,g_j}$ and $\varphi_{ij}=\varphi_{g_i,g_j}$.
	The family of $\mathcal D^r$ submanifolds and transition maps $\{(U_i)_{1 \leq i \leq d_G+1}, (\varphi_{ij})_{1 \leq i,j \leq d_G+1}\}$ defines a $\mathcal D^r$ manifold $Y$.
	We define the $\mathcal D^r$ map $\kappa_i:V_i \to U_i$ by $\kappa_i(x)=g_i\tau(g_i^{-1}x)$.
	It is obvious that, for $1 \leq i,j \leq d_G+1$ and $x \in V_i \cap V_j$, the equality $\kappa_j(x)=\varphi_{ij}(\kappa_i(x))$ holds true.
	Therefore, the family $(\kappa_i)_{1 \leq i \leq d_G+1}$ defines a $\mathcal D^r$ map $\kappa:G \to Y$.
	It is obvious that $\kappa(g)=\kappa(h)$ if and only of $g^{-1}h \in H$.
	
	We want to prove that $Y$ is definably normal.
	Take a definable closed subset $C$ and a definable open subset $O$ of $Y$ with $C \subseteq O$.
	Since $G$ is definably normal, there exists a definable open subset $W$ of $G$ such that $\kappa^{-1}(C) \subseteq W \subseteq \mycl_G(W) \subseteq \kappa^{-1}(O)$.
	The set $\bigcup_{h \in H}h \mycl_G(W)$ is closed in $G$ by \cite[Proposition 1.10]{M} because it is the image of the bounded closed definable set $H \times \mycl(W)$ under the multiplication in $G$.
	Therefore, we may assume that $W$ is $H$-invariant by replacing $W$ with $\bigcup_{h \in H}hW$ because $\kappa^{-1}(C) \subseteq \bigcup_{h \in H}hW\subseteq \mycl_G(\bigcup_{h \in H}hW) \subseteq \bigcup_{h \in H}h\mycl_G(W) \subseteq\kappa^{-1}(O)$.
	It is easy to check that the image of a definable $H$-invariant open subset of $G$ under $\kappa$ is open by  the definition of $\kappa$.
	Therefore, $\kappa(W)$ is open in $Y$.
	The image $\kappa(\mycl_G(W))$ is closed in $Y$ because $\kappa(\mycl_G(W))$ is $H$-invariant and $\kappa$ is surjective.
	
	A definably normal $\mathcal D^r$ manifold is imbeddable as a $\mathcal D^r$ submanifold by Theorem \ref{thm:imbedding}.
	Let $\omega:Y \to X$ be the imbedding.
	Set $\iota=\omega \circ \kappa$.
	Note that the restriction $\iota|_{U_i}$ is a $\mathcal D^r$ diffeomorphism by the definition of $\iota$.
	We show that $\iota$ satisfies the requirements of the theorem.
	The non-trivial part is that $\iota$ is definably identifying.
%	Let $\beta:(0,\varepsilon) \to X$ be a definable curve completable in $X$.
%	Set $x=\lim_{x \to 0}\beta(t)$.
%	There exists $1 \leq i \leq d_G+1$ such that $x \in \iota(U_i)$.
%	Taking a smaller $\varepsilon>0$ if necessary, we may assume that the definable curve $\beta$ is contained in $\iota(U_i)$ because $\iota(U_i)$ is open in $X$.
%	Set $\alpha=(\iota|_{U_i})^{-1}\circ \beta:(0,\varepsilon) \to G$ and $g=(\iota|_{U_i})^{-1}(x)$.
%	We obviously have $\beta = \iota \circ \alpha$ and $\lim_{t \to 0} \alpha(t)=g$.
%	It implies that $\iota$ is definably identifying by Lemma \ref{lem:identifying_eq}.
	Let $K$ be a definable subset of $X$ such that $\iota^{-1}(K)$ is closed in $G$.
	Since $G$ is bounded and closed in the ambient space $F^n$, its subset  $\iota^{-1}(K)$ is closed and bounded in $F^n$.
	The image $K=\iota(\iota^{-1}(K))$ is closed and bounded in the ambient space of $X$ by \cite[Proposition 1.10]{M}.
	It implies that $K$ is closed in $X$.
	
	The final task is to demonstrate the `in addition' part.
	We assume that $H$ is a normal subgroup.
	We only demonstrate that the multiplication $\mymult_{G/H}$ satisfies the conditions given in the theorem.
	The proof for the inverse $\myinv_{G/H}$ is similar.
	We omit it.
	Let $\mymult_G:G \times G \to G$ be the multiplication in $G$.
	For any $1 \leq i,j,k \leq d_G+1$, set $W_{ijk}:=(\iota \times \iota)((U_i \times U_j) \cap (\mymult_G)^{-1}(V_k)) \subseteq X \times X$.
	The family $\{W_{ijk}\}_{1 \leq i,j,k \leq d_G+1}$ is a definable open cover of $X \times X$.
	We define $\mymult_{G/H}: X \times X \to X$ as follows:
	Let $(g,h) \in X \times X$.
	There exists $1 \leq i,j,k \leq d_{G}+1$ such that $(g,h) \in W_{ijk}$.
	Set $$\mymult_{G/H}(g,h) = \iota(\mymult_G((\omega|_{U_i})^{-1}(g), (\omega|_{U_j})^{-1}(h))).$$
	The right hand of the above equation is independent of the choice of $1 \leq i,j,k \leq d_{G}+1$ with $(g,h) \in W_{ijk}$ because $H$ is a normal subgroup.
	Here, we identify the definable subset $U_i$ in $Y$ with the definable subset $g_iU$ in $G$.
	It is a $\mathcal D^r$ map and satisfies the equality $\mymult_{G/H}(\iota(g_1),\iota(g_2)) =\iota(g_1\cdot g_2)$ for $g_1,g_2 \in G$.
\end{proof}

\end{document}